\documentclass[10pt, letterpaper,reqno]{amsart}
\usepackage{amssymb,a4wide,amsthm,amsmath,amsfonts,xcolor}

\usepackage{mathrsfs}
\usepackage{amscd}
\usepackage{marginnote}
\usepackage{verbatim}
\usepackage[active]{srcltx}
\usepackage[colorlinks,linkcolor={red},
citecolor={red},urlcolor={red},]{hyperref}
\usepackage{upref}
\usepackage{hyperref}
\usepackage{enumerate}
\usepackage{bbm}

\theoremstyle{plain}
\newtheorem{thm}{Theorem}[section]
\theoremstyle{plain}

\newtheorem{lem}[thm]{Lemma}
\newtheorem{prop}[thm]{Proposition}
\newtheorem{cor}[thm]{Corollary}

\theoremstyle{definition}
\newtheorem{defi}[thm]{Definition}

\numberwithin{equation}{section}

\newcommand{\lb}{\langle}
\newcommand{\rb}{\rangle}

\newcommand\addbnew[1]{{\color{black} #1}}
\def\epsilon{\varepsilon}

\def\v{\mathbf{v}}

\def\u{\mathbf{u}}

\def\Hh{\mathbb{H}^1}

\def\R3{\mathbb{R}^3}
\def\RN{\mathbb{R}^N}
\def\L{\mathbb{L}^2}

\newcommand{\eps}{\varepsilon}

\newcommand{\A}{\mathcal{A}}
\newcommand{\Rr}{\mathbb{R}}
\newcommand{\Na}{\mathbb{N}}

\newcommand{\M}{\mathbf{M}}

{%
\setcounter{enumi}{0}

\begin{enumerate}}
{\end{enumerate} }

{
\setcounter{enumi}{0}

\begin{enumerate}}
{\end{enumerate} }

\def\e{{\text{e}}}

\numberwithin{equation}{section} \allowdisplaybreaks

\newcommand{\rA}{\mathrm{ A}}
\newcommand{\rV}{\mathrm{ V}}

\begin{document}

\title[LLGEs: controllability by Low Modes Forcing]{Landau-Lifshitz-Gilbert equations: Controllability by Low Modes Forcing for deterministic version and Support Theorems for Stochastic version}

\author[Mrinmay Biswas]{Mrinmay Biswas}

\address{%
   Department of Mathematics,
	Montanuniversit\"at Leoben,
	Austria.}
\email{mrinmay.biswas@unileoben.ac.at}

\author{Erika Hausenblas}
   \address{%
   Department of Mathematics,
	Montanuniversit\"at Leoben,
	Austria.}
\email{erika.hausenblas@unileoben.ac.at}

\author[Debopriya Mukherjee]{Debopriya Mukherjee}

\address{%
   Department of Mathematics,
	Montanuniversit\"at Leoben,
	Austria.}
\email{debopriya.mukherjee@unileoben.ac.at}

\thanks{The second author of the paper is supported by Austrian Science
	Foundation, project number P 32295. The third author is supported by
	Marie Sk{\l}odowska-Curie Individual Fellowships H2020-MSCA-IF-2020,
	888255. 
}

\date{\today}

\begin{abstract}
In this article, we study the controllability issues of the Landau-Lifshitz-Gilbert Equations (LLGEs), accompanied with non-zero exchange energy only, in an interval in one spatial dimension with Neumann boundary conditions. The paper is of twofold. In the first part of the paper, we study the controllability issues of the LLGEs. The control force acting here is degenerate i.e., it acts through a few numbers of low mode frequencies. We exploit the Fourier series expansion of the solution. We borrow methods of differential geometric control theory (Lie bracket generating property) to establish the global controllability of the finite-dimensional Galerkin approximations of LLGEs. We show $L^2$ approximate controllability of the full system. In the second part, we consider the LLGEs with lower-dimensional degenerate random forcing (finite-dimensional Brownian motions) and study support theorems.

\end{abstract}


\maketitle



\section{Introduction}
Let $\mathcal{O}\subset\mathbb R^d \, (1 \leq d \leq 3)$ be an open bounded region and $\mathbb{S}^2$ represents the two dimensional unit sphere in $\R3$. The magnetic behavior of any ferromagnetic material is determined by the special property called magnetization $\M:[0,\infty) \times \mathcal{O} \rightarrow \mathbb{S}^2$, with the assumption that the material is saturated at the initial time, i.e.,
\begin{align} \label{Mnorm1}
|\M_0(x)|_{\R3}=1 \quad \mbox{for a.e.}\,\, x \in  \mathcal{O}.
\end{align}
The scientific theory illustrating the magnetic behavior in ferromagnetic materials has been initiated by Weiss \cite{Weiss_1907}, and further developed by Landau and Lifshitz \cite{LanLif} and Gilbert \cite{Gil}. The magnetization $\M$ at elevated temperature which is below the critical temperature (called Curie temperature) satisfies 
the following Landau-Lifshitz-Gilbert Equations (LLGEs):
\begin{align} \label{eq.M}
\frac{\partial \M}{\partial t}(t, x) =\mu_1 \M(t,x) \times H_{eff}(t,x) -\mu_2 \M(t,x) \times (\M(t,x) \times H_{eff}(t,x)),\,\,t>0,\,\, x \in \mathcal{O},
\end{align}
where $\times$ is the vector cross product in $\R3$, and the parameters $\mu_1 \neq 0,\mu_2>0$ are constants. If $\mathcal{E}$ is the total energy of the magnetic energy functional, then, $$H_{eff}= H_{eff}(\M): = - \nabla_{\M} \mathcal{E}(\M)$$ is the so-called effective field which acts on spins in the ensemble. This total energy $\mathcal{E}$ is the sum of the exchange energy $\mathcal{E}_{ex}(\M)$, anisotropy energy $\mathcal{E}_{an}(\M)$,  and other energies, see Visintin \cite{Visintin_1985}.  
 There are evidences that a large part of the mathematical difficulty of the
problem seems to stem from the exchange energy. In such situation, when the energy
functional consists of the exchange energy only,
$\mathcal{E}=\frac{1}{2} \int_\mathcal{O} | \nabla\M(x)|^2 \,dx,$ we have $H_{eff} = \Delta\M$
and we obtain the following version of the LLGEs:
\begin{equation} \label{M1}
\left\{\begin{array}{ll}
\frac{\partial \M}{\partial t}(t,x) &=\mu_1 \M(t,x) \times \Delta\M(t,x) -\mu_2 \M(t,x) \times (\M(t,x) \times \Delta\M(t,x)),\,\, t>0,\,\, x\in \mathcal{O}, \\
\frac{\partial \M}{\partial \nu}(t,x) & =0,\,\, t>0,\,\, x\in\partial \mathcal{O},\\
\M(0,x)&=\M_0(x),\,\, x \in \mathcal{O}.
\end{array}\right.
\end{equation}

If both the exchange and anisotropy energies are present in the system, the total energy $\mathcal{E}$ of the LLGEs is of the following form
$$\mathcal{E}(\M) = \mathcal{E}_{an}(\M) + \mathcal{E}_{ex}(\M) =\int_\mathcal{O} \Big(\psi(\M(x))+\dfrac{1}{2} | \nabla\M(x)|^2 \Big)\,dx$$ where $\mathcal{E}_{an}(\M) :=\int_\mathcal{O} \psi(\M(x))\, dx$ represents the anisotropy energy, the energy which is directionally dependent. The exchange energy is given by $\mathcal{E}_{ex}(\M) := \int_\mathcal{O} |\nabla\M(x)|^2 dx$,
and the effective field $H_{eff}$ takes the form $\Delta\M - \nabla\psi(\M).$
\par
The LLGEs describe the behavior of the magnetic property of the ferromagnetic materials. These materials are used in electro-magnetic storage devices such as memory cards, computer hard disks, debit and credit cards, audio-video tape recordings, large scientific computing devices, etc. to store digital data contents. Each data content is stored in a memory device as a uniquely assigned stable magnetic state of the ferromagnet. For the computational purpose, it is desirable to move from one stable state to another stable state by using minimum requirements as well as not affecting the data which is already stored in it. The magnetic state of a ferromagnet can be changed by changing the magnetic field, $H_{eff},$ which is viewed as the control (Carbou et al. \cite{Car1, Car2, Car3}, Alouges-Beauchard \cite{AB}, Noh 
et al. \cite{Noh}).
According to the author's knowledge, there are only a few results on the control of magnetization described by the LLGEs. From the control-theoretic point of view, most of the articles are either experimental or they are numerically simplified. In Alouges-Beauchard \cite{AB}, the controlled LLGEs are considered on a
special domain which allows the spatial variable to be fixed, and hence the control problem simplifies to an ordinary differential equation. Experiments demonstrating the control of domain walls in a nano-wire are presented in Noh et al. \cite{Noh}. Numerical simulations have also been conducted
on the control of domain walls in nanowires (Wieser et al. \cite{Wies}).

In this article, we study the controllability properties of the LLGEs in the one-dimensional spatial domain $\mathcal{O}=(0,2\pi)$ by means of low mode forcing. We consider the following deterministic controlled LLGEs
\begin{align}\label{M1.N} 
\left\{\begin{array}{ll} 
\M_t(t,x) & =\mu_1 \M(t,x) \times \M_{xx}(t,x) -\mu_2 \M(t,x) \times (\M(t,x) \times \M_{xx}(t,x))+ \M(t,x) \times\v(t,x),  \\
& \qquad \qquad\qquad\qquad\qquad\qquad\qquad\qquad \qquad\qquad\qquad t\in (0,T],\,\, x\in (0,2\pi), \\
\M_x(t,0) & =0=\M_x(t,2\pi),\,\, t\in (0,T],  \\
\M(0,x) & =\M_0(x),\,\,x \in(0,2\pi),
\end{array}\right.
\end{align}
with  
\begin{equation*}
\M_0\in\mathbb{H}^1(0,2\pi;\mathbb{S}^2):=\Bigl\{\M \in \mathrm{H}^1(0,2\pi; \mathbb R^3) \mbox{ such that }|\M(x)|_{\R3}=1  \mbox{ for a.a. } x \in (0,2\pi)  \Bigr\},
\end{equation*}
where $\v$ is the control at our disposal. We are interested to show the controllability of the above nonlinear system by using a control as degenerate forcing. Here, degenerate means $\v$ lies in a finite dimensional subspace $\hbox{span}\{\varphi_ke_j\,:\, (k,j)\in\mathcal{K}^1\}$ for some finite subset $\mathcal{K}^1\subset \mathbb{N}_0\times \{1,2,3\},$ i.e., $\v$ 
is acting upon the system through the finite number of Fourier modes with possibly lower frequencies.
\par 
The paper is twofold. Firstly, we aim to establish the global controllability of the Galerkin approximated system of \eqref{M1.N}. Next, we prove $\mathbb{L}^2$-approximate controllability of the 
system \eqref{M1.N}. In the second part of the paper, we consider the system \eqref{M1.N} with lower-dimensional degenerate random forcing (finite-dimensional Brownian motions) and study support theorems.
\par
\subsection{Controllability of deterministic LLGEs  by low mode forcing}
We shall study the following two controllability aspects:
\begin{enumerate}[(i)]
\item
Global controllability of Galerkin approximation system (finite dimensional);
\item
$L^2$-approximate controllability of the original system (infinite dimensional).
\end{enumerate} 
In system \eqref{M1.N}, control $\v$ is degenerate, i.e., $\v$ is a finite sum of harmonics. In particular, $\v$ is of the form (see \cite{Agra_Sary.3} and references cited therein) 
\begin{align}\label{control.1}
\v(t,x)=\sum_{(k,j)\in \mathcal{K}^1}v_k^j(t)\varphi_k(x) e_j,\quad t\in [0,T],\,\,x\in [0,2\pi],
\end{align}
where $\{\varphi_k: k \in \mathbb N\}$ are eigenfunctions of the Neumann Laplacian in $L^2(0,2\pi),$ 
$\mathcal{K}^1\subset \mathbb{N}_0\times \{1,2,3\}$ is finite subset and 
$v_k^j \in L^\infty(0,T)$ are control functions.
\par
Consider any subset $\mathcal{G} \subset \mathbb{N}_0 \times \{1,2,3\}$ and introduce the Galerkin $\mathcal{G}$-approximation of the system \eqref{M1.N} by projecting this equation onto the linear space spanned by the
harmonics $\{\varphi_k(x)e_j\} $ with $k,j \in \mathbb{N}_0 \times \{1,2,3\}$. It corresponds to keeping in the system \eqref{M1.N} only the equations for the variables $v_k^j$ with $k \in \mathcal{G}$ and changing the condition $k,j \in \mathbb{N}_0 \times \{1,2,3\}$ to $k,j \in \mathcal{G}$.
We now provide formal definition of controllability and approximate controllability.
\begin{defi}
A Galerkin $\mathcal{G}$-approximation of the system \eqref{M1.N} is globally controllable if for any two points $\mathbf{m}_0, \mathbf{m}_1\in \hbox{span}\{\varphi_k e_j: (k,j) \in \mathcal{G}\}$, there exists $T>0$ and a control $\v$ of the form \eqref{control.1} which steers the solution of the Galerkin $\mathcal{G}$-approximation from 
$\mathbf{m}_0$ to $ \mathbf{m}_1$ in time $T$. The Galerkin $\mathcal{G}$-approximation of the system  \eqref{M1.N} is time-$T$ globally controllable if $T$ can be chosen the same for all $\mathbf{m}_0, \mathbf{m}_1\in \hbox{span}\{\varphi_k e_j: (k,j) \in \mathcal{G}\}$.
\end{defi}

We will consider a Galerkin $\mathcal{G}$-approximation of the system \eqref{M1.N} which is the system \eqref{inf.sys.g1}-\eqref{inf.sys.g3} given in Section \ref{cont.5}.

\begin{thm}\label{glo.cont.1}
Let $T > 0$. Then, the Galerkin $\mathcal{G}$-approximation \eqref{inf.sys.g1}-\eqref{inf.sys.g3}
 of the system \eqref{M1.N} is time-$T$ globally controllable.
\end{thm}
Theorem \ref{glo.cont.1} provides the result of global exact controllability of the Galerkin $\mathcal{G}$-approximation of the system \eqref{M1.N}. Since, Galerkin $\mathcal{G}$-approximation is a finite dimensional approximation of the main system \eqref{M1.N}, a natural question arises about the controllability of the original system \eqref{M1.N}. We prove that the system \eqref{M1.N} is approximately controllable in $\mathbb{L}^2$-topology.
\begin{defi}[$\mathbb{L}^2$-approximately controllability]
The system \eqref{M1.N} is globally $\mathbb{L}^2$-approximately controllable, if for any $\M_0,\M_1 \in \mathbb{H}^1(0,2\pi;\mathbb{S}^2)$ and any $\epsilon>0$, there exists a $T > 0$ and a control $\v \in \mathbb{L}^2$ which steers the system \eqref{M1.N} from $\M$ to the $\epsilon$ neighbourhood of $\M_1$ in the $\mathbb{L}^2$ norm in time $T$.
\end{defi}
We have the following $\mathbb{L}^2$-approximate controllability
results.
\begin{thm}\label{L2.app.cont.1}
Let $T > 0$. Then, the system  \eqref{M1.N}
is time-$T$ globally $\mathbb{L}^2$-approximate controllable.
\end{thm}
We will prove Theorem \ref{glo.cont.1} and Theorem \ref{L2.app.cont.1} in Section \ref{cont.5}. However, we briefly outline the sketch of the main ideas of the proof below. To prove the Galerkin $\mathcal{G}$-approximated system \eqref{inf.sys.g1}-\eqref{inf.sys.g3} is globally exact controllable, we apply Chow-Rashevskii Theorem (see Theorem \ref{ChowR} in Appendix \ref{app geo cont}). To do so, we need to find an appropriate set $\mathcal{K}^1$ with minimum number of control modes such that the control vector fields have Lie bracket generating property. 
\par 
To prove Theorem \ref{L2.app.cont.1}, we first decompose the original system into two parts. The first one is a finite-dimensional system where the dimension corresponds to a large but finite number of Fourier modes and the rest is an infinite-dimensional system that corresponds to the remaining Fourier modes. Now, by Theorem \ref{glo.cont.1}, the finite-dimensional system is globally exact controllable. Using apriori $\mathbb{L}^2$ estimate of the solution of the system \eqref{M1.N}  and Theorem \ref{glo.cont.1}, we prove that the system \eqref{M1.N} is  $\mathbb{L}^2$-approximately controllable.
\par
Next, we consider the 
stochastic LLGEs system \eqref{sM1.N}. By the above controllability results of a deterministic system and an application of Girsanov transformation, we prove a support theorem for the system \eqref{sM1.N}.
\subsection{Support theorems for stochastic LLGEs}
In this article, we study small ball probabilities for the degenerate (forcing) noise term of the stochastic LLGEs. Let $(\Omega, \mathcal{F}, (\mathcal{F}_t)_{t\geq 0}, \mathbb{P})$ be a filtered probability space where the filtration $(\mathcal{F}_t)_{t\geq 0}$ satisfies the usual conditions i.e.,
\begin{itemize}
\item[(i)] $\mathbb{P}$ is complete on $(\Omega, \mathcal{F})$,
\item[(ii)] for each $t\geq 0$, $\mathcal{F}_t$ contains all $(\mathcal{F},\mathbb{P})$-null sets and
 \item[(iii)] the filtration $(\mathcal{F}_t)_{t\geq 0}$ is right-continuous.
  \end{itemize}
We consider the following system
\begin{align} \label{sM1.N}
\left\{\begin{array}{ll} 
d\M(t,x) & =(\mu_1 \M(t,x) \times \M_{xx}(t,x) -\mu_2 \M(t,x) \times (\M(t,x) \times \M_{xx}(t,x)))dt   \\
& \qquad \qquad\qquad\qquad +  \sum_{(k,j)\in \mathcal{K}^1}\left(\M(t,x) \times (\varphi_k(x)  e_j)\right)\circ d\beta_k^j(t), \quad t\in (0,T],\,\, x\in (0,2\pi),  \\
\M_x(t,0) & =0=\M_x(t,2\pi),\,\, t\in (0,T],  \\
\M(0,x) & =\M_0(x),\,\,x \in(0,2\pi) ,
\end{array}\right.
\end{align}
with Neumann boundary condition, 
$\M_0\in \mathbb{H}^1(0,2\pi;\mathbb{S}^2)$, 
$\mathcal{K}^1\subset \mathbb{N}_0\times \{1,2,3\}$ is finite subset and 
$\beta_k^j(t)$ are independent real valued Brownian motions with respect to the filtration $(\mathcal{F}_t)_{t\geq 0}$. 
Our main results in this directions are the following:

\begin{thm}\label{s.glo.cont.1}
 Let $T > 0$ and $K\in\Na_0$. Let us consider the set $$
 \mathcal{G}_K:=\{(k,l)\,|\,k\leq K,\, l=1,2,3\}\subset \mathbb{N}_0\times \{1,2,3\}
\quad{and} \quad  
S_K:=\mbox{span}\{\varphi_ie_j|\,
(i,j)\in\mathcal{G}_K\}\subset \Hh.$$ 
Then, for any $R, \eps>0$, and, for any two points $\mathbf{m}_0, \mathbf{m}_1\in S_K$, there exists $\delta>0$ such that the following holds:
\newline
 for any $\widetilde{\mathbf{m}}_0 \in 
\Big\{ \mathbf{m}: |\mathbf{m}- \mathbf{m}_0|_K \leq R \Big\}
$, let $\mathbf{m}(\cdot,\widetilde{\mathbf{m}}_0)$ be the solution of the Galerkin $\mathcal{G}_K$-approximation \eqref{s.inf.sys.g1}-\eqref{s.inf.sys.g3} with $\mathbf{m}(0,\widetilde{\mathbf{m}}_0)=\widetilde{\mathbf{m}}_0$ being the initial condition. Then, there holds
$$\mathbb{P}\Big(|\mathbf{m}(T,\widetilde{\mathbf{m}}_0)-\mathbf{m}_1|_K \leq \eps\Big)  >\delta. $$
\end{thm}

\begin{thm}\label{s.L2.app.cont.1}
Let $\M$ be the solution of the system \eqref{sM1.N}.
Then, for any $T > 0,$ $\M_0,\M_1\in \mathbb{H}^1(0,2\pi;\mathbb{S}^2)$ and  $\epsilon>0$, there exists $\delta=\delta(\M_0,\M_1,\eps,T)>0$ such that
$$\mathbb{P}(|\M(T)-\M_1|_{\mathbb{L}^2}<\epsilon)>\delta.$$
\end{thm}
\subsection{Novelty of our work}
We prove the exact controllability of Galerkin approximations (system \eqref{inf.sys.g1}-\eqref{inf.sys.g3}) of the system
\eqref{M1.N}. As a next step, we obtain the $\mathbb{L}^2$-approximate controllability of the full system \eqref{M1.N}. To prove the exact controllability of the Galerkin approximated system \eqref{inf.sys.g1}-\eqref{inf.sys.g3}, we first verify the Lie-bracket generating property of the finite-dimensional system.
 The control acts on the system in such a way that while computing the iterated Lie-bracket, we must show that the system satisfies Lie-bracket generating property. Therefore, it is technically challenging to guess the vector fields (in our case $\mathcal{K}^1$, see Section \ref{cont.5}).  
 We prove support theorems for the system \eqref{sM1.N}.

\section{Wellposedness of the LLGEs}
In this section, we discuss the well-posedness results of the systems \eqref{M1.N} and \eqref{sM1.N}. We refer to \cite{BMM} and references cited therein for details about the existence and uniqueness of the solution of the systems \eqref{M1.N} and \eqref{sM1.N}.

\subsection{Required function spaces and operators}
Let us denote the space of Lebesgue measurable real valued square integrable
functions defined on $(0,2\pi)$ by $\mathrm{L}^2(0,2\pi)$ and  
$$\mathbb{L}^2:= (\mathrm{L}^2(0,2\pi))^3=\mathrm{L}^2(0,2\pi;\mathbb{R}^3).$$
Similarly, we denote
$$
\mathrm{H}^1(0,2\pi):=\Big\{u\in L^2(0,2\pi)| \,\, u'\in L^2(0,2\pi) \Big\}\quad \mbox{and}\quad
\mathbb{H}^1:=(\mathrm{H}^1(0,2\pi))^3=\mathrm{H}^1(0,2\pi; \mathbb R^3).
$$
We denote the Sobolev space $\mathbb{H}^1(0,2\pi;\mathbb{S}^2)$ by
\begin{equation}\label{eqn-H^1(D,S^2)}
\mathbb{H}^1(0,2\pi;\mathbb{S}^2):=\Bigl\{\M \in \mathbb{H}^1 \mbox{ such that }|\M(x)|_{\R3}=1  \mbox{ for a.e. } x \in (0,2\pi)  \Bigr\}.\end{equation}
In particular, $\mathbb{H}^1(\mathcal{O};\mathbb{S}^2)$ is the set of equivalence classes of all functions belonging to  the Sobolev space
$\mathbb{H}^1$ whose values are in the sphere. 
We define the Laplacian with the Neumann boundary conditions by
\begin{equation}
\label{op.n} \left\{
\begin{array}{ll}
D(\mathcal{A}) &:= \{ \M \in \mathbb{H}^2(0,2\pi;\mathbb{R}^3):\M_x(0)=\M_x(2\pi)=0 \},\cr
\mathcal{A}(\M)&:=- \M_{xx}, \quad \M\in D(\mathcal{A}).
\end{array}
\right.
\end{equation}
Let us consider the Gelfand triple
$$
D(\mathcal{A})\subset \Hh \subset \L\cong (\L)' \subset
(\Hh)' \subset (D(\mathcal{A}))',
$$
where  $(\mathbb{H}^1)'$ is the dual space of $\mathbb{H}^1$ with duality pairing
$\langle \cdot ,\cdot \rangle :=_{(\Hh)'}\langle\cdot ,\cdot \rangle_{\Hh}.$
%
 We note that $\A$ is self-adjoint and non-negative  operator in $\mathbb{L}^2.$ Define $\rA_1:=I+\mathcal{A}.$ We note that $\rV:=Dom(\rA_1^{1/2})$ when endowed with the graph norm coincides with $\mathbb{H}^1.$ 
\subsection{Wellposedness of the system \eqref{M1.N}}
In this subsection we recall the wellposedness of the system \eqref{M1.N}. We have given the definition of weak solution to the system \eqref{M1.N} and then mention the existence, uniqueness and continuous dependence of solution on given data.

\begin{defi} \label{defi.weak1.m} (Weak solution) Let $T>0,\,\v\in \mathrm{L}^2(0,T;\mathbb{L}^2) $ and $\M_0\in \Hh $ be fixed. An element $\M \in \mathrm{L}^2(0,T;D(\mathcal{A}))\cap \mathrm{L}^\infty(0,T;\Hh)$ with
$\M'\in \mathrm{L}^2(0,T;\mathbb{L}^2)$
 is said to be a strong solution of the system \eqref{M1.N} if the following items are satisfied:
\begin{enumerate}[(i)]
\item 
 We have
\begin{align} \label{exist.intm.defs}
\int_0^T |\M(t) \times \M_{xx}(t)|_{\mathbb{L}^2}^2 dt < \infty; 
\end{align}
\item \addbnew{$\M$ satisfies the following saturation condition
\begin{align} \label{eqn-m-saturations}
|\M(t,x)|_{\R3}=1 \quad \mbox{for a.e.}\,\, x \in \mathcal{O}, \mbox{ for all}\,\, t \in [ 0,T];
\end{align}}
\item For all $\phi \in \Hh,$ and for all $t \in [0,T]$ we have
\begin{align} \label{eq.weak1}
\langle \M(t), \phi \rangle_{\mathbb{L}^2}&=\langle \M(0), \phi\rangle_{\mathbb{L}^2}-\mu_1 \int_0^t \int_0^{2\pi} \langle \M_x(s,x), \phi_x(x) \times \M(s,x) \rangle_{\R3} dx \, ds \notag\\
&-\mu_2 \int_0^t \int_0^{2\pi} \langle \M_x(s,x), (\M \times \phi)_x(s,x) \times \M(s,x) \rangle_{\R3} dx\, ds\notag\\
&+\int_0^t \int_0^{2\pi} \langle \M(t,x) \times\v(s,x),\phi(x)\rangle_{\R3} dx\, ds.
\end{align}
\end{enumerate}
\end{defi}
 The following existence result is proved using the Faedo-Galerkin approximation in \cite[Theorem 3.2, Theorem 3.16, Lemma 5.2]{ BMM}, Ba{\~n}as et.al. \cite{BBP1, BBNP1}. 
\begin{thm}  [\cite{BMM}] \label{exist.thmw}
 Let $\M_0 \in \mathbb{H}^1(0,2\pi;\mathbb{S}^2)$ and $\v\in \mathrm{L}^2(0,T;\mathbb{L}^2).$ Then, there exists a weak solution $\M$ to the system \eqref{M1.N}. Moreover,
 we have  $\M \in C ([0,T];\Hh)$ and
 there exists a constant $C> 0$, depending on $T, \mu_1,\mu_2, |\M_0|_{\mathbb{H}^1}$ such that
\begin{align} \label{exist.mw}
\sup_{t \in [0,T]} |\M(t)|_{\mathbb{H}^1} \leq C,
\end{align}
and 
\begin{align} \label{exist.intmw}
\int_0^T |\M(t) \times \M_{xx}(t)|_{\mathbb{L}^2}^2 dt \leq C.
\end{align}
\end{thm}

\begin{thm} [\cite{BMM}] \label{uniq_cont.thmw}
 Let  $\M_i \in \mathrm{L}^4(0,T;\Hh)$ be
solutions to  \eqref{M1.N} with $\M_i(0)=\M_{i0}\in \mathbb{H}^1(0,2\pi;\mathbb{S}^2)$ and $\v_i\in \mathrm{L}^2(0,T;\mathbb{L}^2)$ for i=1,2 .    Then, for $i=1,2,$ $\M_i^\prime \in \mathrm{L}^2(0,T;(\Hh)^\prime)$ and $\M_i$  solve the following equation
\begin{align} 
 (\M_i)_t(t)+\mu_2 \mathcal{A} \M_i(t)&=\mu_2 |\M_i(t)_x|_{\R3}^2 \M_i(t) +\mu_1 (\M_i(t) \times (\M_i(t)_{xx}))
+ \M_i(t,x) \times\v_i(t), \; \; t\in (0,T),\label{eqn.var1}\\
\M_{ix}(t,0) & =0=\M_{ix}(t,2\pi),\,\, t\in (0,T], \label{eqn.var2} \\
\M_i(0,x) & =\M_{i0}(x),\,\,x \in(0,2\pi), \label{eqn.var3}
\end{align}
in the weak sense with respect to the Gelfand triple $\Hh \subset \mathbb{L}^2 \subset (\Hh)^\prime$.
Moreover, there exists a constant $C>0$ and a function $\varphi_C\in L^1(0,T)$, depending on the norms $\|\M_i\|_{\mathrm{L}^4(0,T;\Hh)},\,i=1,2$ , such that the following estimate
\begin{align}\label{eqn.var1.est}
|\M_2(t)-\M_1(t)|^2_{\mathbb{L}^2} \leq \Big(|\M_2(0)-\M_1(0)|^2_{\mathbb{L}^2}+C\int_0^t|\v_2(s)-\v_1(s)|_{\mathbb{L}^2}^2ds\Big) e^{2\int_0^t \varphi_C(s) ds}, \; t\in [0,T]
\end{align}
holds.
\end{thm}
As a consequence of Theorem \ref{uniq_cont.thmw}, we have the following result.
\begin{cor}\label{cont.2}
For given $(\M_0, \v)\in \Hh\times \mathrm{L}^2(0,T;\mathbb{L}^2)$, let $\M\in C([0,T];\Hh)$ be the unique solution of \eqref{M1.N}. Then, the mapping 
$$\mathcal{S}:\Hh\times \mathrm{L}^2(0,T;\mathbb{L}^2)\to 
C([0,T];\mathbb{L}^2)
\quad \hbox{defined by} \quad
 \mathcal{S}(\M_0, \v)=\M $$
is continuous.
\end{cor}
\subsection{Wellposedness of the system \eqref{M1.N}}
In this subsection, we mention the wellposedness of the system \eqref{sM1.N}. 
The following existence result is proved using Faedo-Galerkin approximation in  \cite[Theorem 7.1, Lemma 5.2]{BMM}.
\begin{thm}[\cite{BMM}] \label{exist.thmws}
 Let $\M_0 \in \mathbb{H}^1(0,2\pi;\mathbb{S}^2)$ and $\v\in \mathrm{L}^\infty(0,T;\mathbb{L}^\infty).$ Then, there exists a weak solution $\M$ to the system \eqref{M1.N}. Moreover,  $\M \in C ([0,T];\Hh)$ and
 there exists a constant $C> 0$, depending on $T, \mu_1,\mu_2, |\M_0|_{\mathbb{H}^1}$ such that
\begin{align} \label{exist.mws}
\mathbb{E}\Big[\sup_{t \in [0,T]} |\M(t)|_{\mathbb{H}^1}\Big] \leq C,
\end{align}
and 
\begin{align} \label{exist.intmws}
\mathbb{E}\Big[\int_0^T |\M(t) \times \M_{xx}(t)|_{\mathbb{L}^2}^2 \, dt\Big]\leq C.
\end{align}
\end{thm} 

\section{Controllability of the system \eqref{M1.N} by low mode forcing}\label{cont.5}
In this section, we prove the exact controllability of a Galerkin $\mathcal{G}$-approximation of the system \eqref{M1.N}. Using the Fourier decomposition, we write the system \eqref{M1.N} as an infinite-dimensional system of ordinary differential equation mode (Fourier) by mode (see Appendix \ref{fourierd}). Then, we consider a Galerkin $\mathcal{G}$-approximated system \eqref{inf.sys.g1}-\eqref{inf.sys.g3} of the system \eqref{M1.N} indexed by a finite subset $\mathcal{G}_K \subset \mathbb{N}_0\times \{1,2,3\}.$ The control is acting on the system through the vector fields given
by the set $\mathcal{K}^1:=\{(0,1),(0,2),(1,1)\}\subset \mathcal{G}_K.$ Then, we show that the vector fields corresponding to the set  $\mathcal{K}^1$ are Lie bracket generating (see Definition \ref{liegen} in the Appendix). We conclude our result by applying the Chow-Rashevskii Theorem (see Theorem \ref{ChowR}). Next, we prove $\mathbb{L}^2$-approximate controllability of the 
system \eqref{M1.N}.

\subsection{Galerkin approximations of LLGEs}\label{control.galerkin}
For simplicity of the exposition, we assume $\mu_1=1=\mu_2.$ For any $K\in\Na_0$, let us consider the set $$\mathcal{G}_K:=\{(k,l)\,|\,k\leq K,\, l=1,2,3\}\subset \mathbb{N}_0\times \{1,2,3\}.$$
 We introduce the Galerkin $\mathcal{G}_K$-approximation of
the system \eqref{inf.sys.o1}-\eqref{inf.sys.o3} by projecting the system \eqref{inf.sys.o1}-\eqref{inf.sys.o3} onto the linear subspace $$S_K:=\mbox{span}\{\varphi_ie_j|\,
(i,j)\in\mathcal{G}_K\}\subset \Hh.$$ 
Hence, the resulting $\mathcal{G}_K$-Galerkin control system can be written as (component wise), for $t\in (0,T),$
\begin{align}\label{inf.sys.g1}
\frac{d}{dt}m_i^1(t) & =\sum_{\substack{n+k=i,\\ |n-k|=i}}\lambda_k \left(m_n^2(t)m_k^3(t)-m_n^3(t)m_k^2(t)\right)\notag\\
& \quad - \sum_{\substack{n+k+l=i,\\ |n-k+l|=i,\\ |n+k-l|=i,\\|k+l-n|=i}}
\lambda_k \Big\{m_l^2(t)\left(m_n^1(t)m_k^2(t)-m_n^2(t)m_k^1(t)\right)-m_l^3(t)\left(m_n^3(t)m_k^1(t)-m_n^1(t)m_k^3(t)\right)\Big\}\notag \\
& \quad + \sum_{j=1}^3\sum_{\substack{n+k=i,\\|n-k|=i\\(k,l)\in \mathcal{K}^1}}v_k^l(t)m_n^j(t)\langle e_j\times e_l,e_1\rangle,
\end{align}
\begin{align}\label{inf.sys.g2}
\frac{d}{dt}m_i^2(t) & = \sum_{\substack{n+k=i,\\|n-k|=i}}\lambda_k \left(m_n^3(t)m_k^1(t)-m_n^1(t)m_k^3(t)\right)
\notag\\
& \quad - \sum_{\substack{n+k+l=i,\\|n-k+l|=i,\\ |n+k-l|=i,\\|k+l-n|=i}}
\lambda_k \Big\{m_l^3(t)\left(m_n^2(t)m_k^3(t)-m_n^3(t)m_k^2(t)\right)-m_l^1(t) \left(m_n^1(t)m_k^2(t)-m_n^2(t)m_k^1(t)\right)\Big\}\notag \\
& \quad + \sum_{j=1}^3\sum_{\substack{n+k=i,\\|n-k|=i\\(k,l)\in \mathcal{K}^1}}v_k^l(t)m_n^j(t)\langle e_j\times e_l,e_2\rangle,
\end{align}
and
\begin{align}\label{inf.sys.g3}
\frac{d}{dt}m_i^3(t) & = \sum_{\substack{n+k=i,\\|n-k|=i}}\lambda_k \left(m_n^1(t)m_k^2(t)-m_n^2(t)m_k^1(t)\right)
\notag\\
& \quad - \sum_{\substack{n+k+l=i,\\|n-k+l|=i,\\ |n+k-l|=i,\\|k+l-n|=i}}
\lambda_k  \left\{m_l^1(t)\left(m_n^3(t)m_k^1(t)-m_n^1(t)m_k^3(t)\right) - m_l^2(t) \left(m_n^2(t)m_k^3(t)-m_n^3(t)m_k^2(t)\right)\right\}\notag \\
& \quad + \sum_{j=1}^3\sum_{\substack{n+k=i,\\|n-k|=i\\(k,l)\in \mathcal{K}^1}}v_k^l(t)m_n^j(t)\langle e_j\times e_l,e_3\rangle,
\end{align}
for $(i,j),(k,l),(n,j)\in\mathcal{G}_K.$
We now show the exact controllability of the Galerkin approximated system.
We continue to verify the full Lie rank property for the system \eqref{inf.sys.g1}-\eqref{inf.sys.g3}.
\subsection{Lie bracket}
We now compute the Lie bracket generating property of the vector fields associated with the control $\v.$
Since we need to find an appropriate set $\mathcal{K}^1$ with minimum number of control modes such that the control vector fields have Lie bracket generating property. 
 We start with $\mathcal{K}^1=\{(0,1),(0,2),(1,1)\}$ and show that the control vector fields have Lie bracket generating property. The control forcing term ${{\bf{F}}(t, \M)}$ can be written as: for $t \in [0,T]$ and $\M\in S_K$
\begin{align*}
{\bf{F}}(t,\M):&= v_0^1(t){\bf{f}}^{0,1}(\M)+v_0^2(t){\bf{f}}^{0,2}(\M)+v_1^1(t){\bf{f}}^{1,1}(\M), \\
  &=  v_0^1(t)\begin{pmatrix}
          f_{i,1}^{0,1}(\M)\\ f_{i,2}^{0,1}(\M) \\
          f_{i,2}^{0,1}(\M) \end{pmatrix} +  v_0^2(t)\begin{pmatrix}
          f_{i,1}^{0,2}(\M)\\ f_{i,2}^{0,1}(\M) \\
          f_{i,2}^{0,2}(\M) \end{pmatrix} +  v_1^1(t) \begin{pmatrix}
          f_{i,1}^{1,1}(\M)\\ f_{i,2}^{1,1}(\M) \\
          f_{i,2}^{1,1}(\M) \end{pmatrix}
\end{align*}
where 
\begin{align}\label{cont_vec}
{\bf{f}}^{k,l}(\M)&= \sum_{(i,j)\in\mathcal{G}}\left(\int_0^{2\pi}\langle \M(t,x)\times \varphi_k(x)e_l,\varphi_i(x)e_j  \rangle_{\Rr^3},dx\right)\varphi_ie_j,\notag\\
 :&= \frac{\pi}{2}\sum_{(i,j)\in\mathcal{G}}f_{i,j}^{k,l}(\M) \varphi_ie_j
\end{align} 
is the control vector field corresponding to the vector field $\varphi_i\e_j$ and $$f_{i,j}^{k,l}(\M)=\sum_{p=1}^3 \left( m_{k+i}^p+ m_{|k-i|}^p\right)\langle e_p\times e_l,e_j\rangle.$$
 Therefore, we can write the $i$-th mode component of the control for $i\in \Na_0$ as
\begin{align*}
{\bf{F}}_i(t,\M):&=v_0^1(t)\sum_{p=1}^3 2m_i^p \begin{pmatrix}
      \langle e_p\times e_1,e_1\rangle\\ \langle e_p\times e_1,e_2\rangle \\
      \langle e_p\times e_1,e_3\rangle \end{pmatrix} + 
      v_0^2(t)\sum_{p=1}^3 2m_i^p \begin{pmatrix}
      \langle e_p\times e_2,e_1\rangle\\ \langle e_p\times e_2,e_2\rangle \\
      \langle e_p\times e_2,e_3\rangle \end{pmatrix}  \\
&\qquad\qquad\qquad \qquad\qquad +
      v_1^1(t)\sum_{p=1}^3 \left(m_{i+1}^p+ m_{|i-1|}^p\right) \begin{pmatrix}
      \langle e_p\times e_1,e_1\rangle\\ \langle e_p\times e_1,e_2\rangle \\
      \langle e_p\times e_1,e_3\rangle \end{pmatrix} \\ 
      & = 2v_0^1(t) \begin{pmatrix}
          0\\  m_i^3 \\
          ( -m_i^2) \end{pmatrix} + 
     2 v_0^2(t) \begin{pmatrix}
       (-m_i^3)\\  0 \\
       m_i^1 \end{pmatrix}   +
      v_1^1(t) \begin{pmatrix}
      0\\ m_{i+1}^3+ m_{|i-1|}^3 \\
      -m_{i+1}^2- m_{|i-1|}^2 \end{pmatrix}, \\ 
       :& =v_0^1(t){\bf{f}}_i^{0,1}(\M)+v_0^2(t){\bf{f}}_i^{0,2}(\M)+v_1^1(t){\bf{f}}_i^{1,1},\\
       :& =v_0^1(t)\begin{pmatrix}
          f_{i,1}^{0,1}(\M)\\ f_{i,2}^{0,1}(\M) \\
          f_{i,3}^{0,1}(\M) \end{pmatrix} +v_0^2(t)\begin{pmatrix}
          f_{i,1}^{0,2}(\M)\\ f_{i,2}^{0,2}(\M) \\
          f_{i,3}^{0,2}(\M) \end{pmatrix}+v_1^1(t)\begin{pmatrix}
          f_{i,1}^{1,1}(\M)\\ f_{i,2}^{1,1}(\M) \\
          f_{i,3}^{1,1}(\M) \end{pmatrix}.
\end{align*}
\begin{lem}\label{lie_brac}
The vector field $\{{\bf{f}}^{k,l}(\cdot)\,|\, (k,l)\in\mathcal{K}^1\}$ is Lie bracket generating, where ${\bf{f}}^{k,l}(\cdot)$ is defined in \eqref{cont_vec}.
\end{lem}
\begin{proof}
To compute the Lie bracket, we estimate 
\begin{align*}
\left[{\bf{f}}^{0,1}(\M), {\bf{f}}^{0,2}(\M)\right]_{i,1} & = \sum_{(k,l)\in\mathcal{G}}\left( f_{k,l}^{0,1}(\M)\frac{\partial f_{i,1}^{0,2}(\M)}{\partial m_k^l} - f_{k,l}^{0,2}(\M)\frac{\partial f_{i,1}^{0,1}(\M)}{\partial m_k^l} \right) \\
& =  \sum_{(k,l)\in\mathcal{G}} f_{k,l}^{01}(\M) (-\delta_{k,l}^{i,3}) -0    = -f_{i,3}^{0,1}(\M)\\
&= 2m_i^2,\\
\left[{\bf{f}}^{0,1}(\M), {\bf{f}}^{0,2}(\M)\right]_{i,2}&= \sum_{(k,l)\in\mathcal{G}}\left( f_{k,l}^{0,1}(\M)\frac{\partial f_{i,2}^{0,2}(\M)}{\partial m_k^l} - f_{k,l}^{0,2}(\M)\frac{\partial f_{i,2}^{0,1}(\M)}{\partial m_k^l} \right) \\
&= 0- \sum_{(k,l)\in\mathcal{G}} f_{k,l}^{0,2}(\M)\delta_{k,l}^{i,3}  =  -f_{i,3}^{0,2}(\M)    \\
&= - 2m_i^1,\\
\left[{\bf{f}}^{0,1}(\M), {\bf{f}}^{0,2}(\M)\right]_{i,3}&= \sum_{(k,l)\in\mathcal{G}}\left( {\bf{f}}_{k,l}^{0,1}(\M)\frac{\partial {\bf{f}}_{i,3}^{0,2}(\M)}{\partial m_k^l}-{\bf{f}}_{k,l}^{0,2}(\M)\frac{\partial {\bf{f}}_{i,3}^{0,1}(\M)}{\partial m_k^l} \right)
= 0.
\end{align*}
Therefore, by above computation we note that $\left[{\bf{f}}^{0,1}(\M), {\bf{f}}^{0,2}(\M)\right]={\bf{f}}^{0,3}(\M)$. Again we note that
\begin{align*}
\left[{\bf{f}}^{1,1}(\M), {\bf{f}}^{0,2}(\M)\right]_{i,1} & = \sum_{(k,l)\in\mathcal{G}}\left( f_{k,l}^{1,1}(\M)\frac{\partial f_{i,1}^{0,2}(\M)}{\partial m_k^l} - f_{k,l}^{0,2}(\M)\frac{\partial f_{i,1}^{1,1}(\M)}{\partial m_k^l} \right) \\
& =  \sum_{(k,l)\in\mathcal{G}} f_{k,l}^{1,1}(\M) (-\delta_{k,l}^{i,3})- 0 =- f_{i,3}^{1,1}(\M)\\
&= m_{i+1}^2+ m_{|i-1|}^2,\\
\left[{\bf{f}}^{1,1}(\M), {\bf{f}}^{0,2}(\M)\right]_{i,2}&= \sum_{(k,l)\in\mathcal{G}}\left(f_{k,l}^{1,1}(\M)\frac{\partial f_{i,2}^{0,2}(\M)}{\partial m_k^l} -  f_{k,l}^{0,2}(\M)\frac{\partial f_{i,2}^{1,1}(\M)}{\partial m_k^l} \right) \\
&= 0- \sum_{(k,l)\in\mathcal{G}} f_{k,l}^{0,2}(\M)(\delta_{k,l}^{i+1,3}+\delta_{k,l}^{|i-1|,3})  =  - f_{i+1,3}^{0,2}(\M)- f_{|i-1|,3}^{0,2}(\M)    \\
&= - m_{i+1}^1- m_{|i-1|}^1,
\end{align*}
and
\begin{align*}
\left[{\bf{f}}^{1,1}(\M), {\bf{f}}^{0,2}(\M)\right]_{i,3}= \sum_{(k,l)\in\mathcal{G}}\left( {\bf{f}}_{k,l}^{1,1}(\M)\frac{\partial {\bf{f}}_{i,3}^{0,2}(\M)}{\partial m_k^l}-{\bf{f}}_{k,l}^{0,2}(\M)\frac{\partial {\bf{f}}_{i,3}^{1,1}(\M)}{\partial m_k^l} \right)
= 0.
\end{align*}
Therefore, by above computation we note that $\left[{\bf{f}}^{1,1}(\M), {\bf{f}}^{0,2}(\M)\right]={\bf{f}}^{1,3}(\M)$. similarly,
we get $\left[{\bf{f}}^{1,1}(\M), {\bf{f}}^{0,3}(\M)\right]= -{\bf{f}}^{1,2}(\M)$. More generally, we can compute,
\begin{align}
\left[{\bf{f}}^{p,1}(\M), {\bf{f}}^{q,2}(\M)\right] & ={\bf{f}}^{|p-q|,3}(\M)+{\bf{f}}^{p+q,3}(\M),\label{lie1} \\
\left[{\bf{f}}^{p,2}(\M), {\bf{f}}^{q,3}(\M)\right] & ={\bf{f}}^{|p-q|,1}(\M)+{\bf{f}}^{p+q,1}(\M),\label{lie2}\\
\left[{\bf{f}}^{p,3}(\M), {\bf{f}}^{q,1}(\M)\right] & ={\bf{f}}^{|p-q|,2}(\M)+{\bf{f}}^{p+q,2}(\M)\label{lie3}.
\end{align}
Indeed,
\begin{align*}
\left[{\bf{f}}^{p,1}(\M), {\bf{f}}^{q,2}(\M)\right]_{i,1} & = \sum_{(k,l)\in\mathcal{G}}\left( f_{k,l}^{p,1}(\M)\frac{\partial f_{i,1}^{q,2}(\M)}{\partial m_k^l} -  f_{k,l}^{q,2}(\M)\frac{\partial f_{i,1}^{p,1}(\M)}{\partial m_k^l} \right), \\
&= \sum_{(k,l)\in\mathcal{G}} f_{k,l}^{p,1}(\M)\frac{\partial f_{i,1}^{q,2}(\M)}{\partial m_k^l} -0 , \\
& = -\sum_{(k,l)\in\mathcal{G}} f_{k,l}^{p,1}\left(\delta_{k,l}^{q+i,3}+\delta_{k,l}^{|q-i|,3}\right),\\
& = - f_{q+i,3}^{p,1}-f_{|q-i|,3}^{p,1},\\
& = m_{p+q+i}^2+m_{|p-q-i|}^2+m_{p+|q-i|}^2+m_{|p-|q-i||}^2,\\
& = m_{p+q+i}^2+m_{|p+q-i|}^2+m_{|p-q|+i}^2+m_{||p-q|-i|}^2,\\
& = f_{i,1}^{p+q,3}(\M)+ f_{i,1}^{|p-q|,3}(\M).
\end{align*}
Similarly,
\begin{align*}
\left[{\bf{f}}^{p,1}(\M), {\bf{f}}^{q,2}(\M)\right]_{i,2} & = \sum_{(k,l)\in\mathcal{G}}\left( f_{k,l}^{p,1}(\M)\frac{\partial f_{i,2}^{q,2}(\M)}{\partial m_k^l} -  f_{k,l}^{q,2}(\M)\frac{\partial f_{i,2}^{p,1}(\M)}{\partial m_k^l} \right), \\
& = -m_{p+q+i}^1-m_{|p+q-i|}^1-m_{|p-q|+i}^1-m_{||p-q|-i|}^1,\\
& = f_{i,2}^{p+q,}(\M)+ f_{i,2}^{|p-q|,3}(\M)
\end{align*}
and
\begin{align*}
\left[{\bf{f}}^{p,1}(\M), {\bf{f}}^{q,2}(\M)\right]_{i,3} & = \sum_{(k,l)\in\mathcal{G}}\left( f_{k,l}^{p,1}(\M)\frac{\partial f_{i,3}^{q,2}(\M)}{\partial m_k^l} -  f_{k,l}^{q,2}(\M)\frac{\partial f_{i,3}^{p,1}(\M)}{\partial m_k^l} \right), \\
& = 0,\\
& = f_{i,3}^{p+q,}(\M)+ f_{i,3}^{|p-q|,3}(\M)
\end{align*}
Hence, \eqref{lie1} follows, by similar calculation \eqref{lie2}, \eqref{lie3} can be proved.  
Therefore, starting with the vector fields $\{{\bf{f}}^{k,l}(\cdot)\,|\, (k,l)\in\mathcal{K}^1\}$ and using mathematical induction and the formulas \eqref{lie1}-\eqref{lie3}, we can show that for each $\M=(m_i^j)_{(i,j)\in\mathcal{G}},$ the finite dimensional vector space Lie$_M$ where  
\begin{align*}
\mbox{Lie}_\M &:=\mbox{span}\Bigg\{\left[{\bf{f}}^{0,1}(\M), {\bf{f}}^{0,2}(\M)\right], \left[{\bf{f}}^{1,1}(\M), {\bf{f}}^{0,2}(\M)\right], \left[{\bf{f}}^{0,1}(\M), \left[{\bf{f}}^{1,1}(\M), {\bf{f}}^{0,2}(\M)\right]\right], \\
&\qquad\qquad\qquad  \left[{\bf{f}}^{1,1}(\M),\left[{\bf{f}}^{0,1}(\M), {\bf{f}}^{0,2}(\M)\right]\right],\dots\Bigg\}
\end{align*}
is same as the finite dimensional subspace 
\begin{align*}
&\mbox{span}\Big\{ \left[{\bf{f}}^{p,1}(\M), {\bf{f}}^{q,2}(\M)\right], \left[{\bf{f}}^{p,2}(\M), {\bf{f}}^{q,3}(\M)\right], \left[{\bf{f}}^{p,3}(\M), {\bf{f}}^{q,1}(\M)\right]\,\Big|\\
&\qquad \qquad\qquad\qquad\qquad\qquad\qquad (p,j),(q,j)\in\mathcal{G},j=1,2,3\Big\},
\end{align*}
i.e. we have
\begin{align*}
\mbox{Lie}_\M  & =\mbox{span}\Big\{ \left[{\bf{f}}^{p,1}(\M), {\bf{f}}^{q,2}(\M)\right], \left[{\bf{f}}^{p,2}(\M), {\bf{f}}^{q,3}(\M)\right], \left[{\bf{f}}^{p,3}(\M), {\bf{f}}^{q,1}(\M)\right]\,\Big|\\
&\qquad \qquad\qquad\qquad\qquad\qquad\qquad (p,j),(q,j)\in\mathcal{G},j=1,2,3\Big\}.
\end{align*}

Hence, the dimension of Lie$_\M=$dimension of $S_{\mathcal{G}}.$ This proves that the vector fields $\{{\bf{f}}^{k,l}(\cdot)\,|\, (k,l)\in\mathcal{K}^1\}$ are Lie bracket generating.
\end{proof}
We, now prove Theorem \ref{glo.cont.1}.
\begin{proof}[Proof of Theorem  \ref{glo.cont.1}]
Since the Galerkin approximated system is Lie-Bracket generating and symmetric, by Chow-Rashevskii Theorem, see Theorem \ref{ChowR}, the Galerkin approximated system is time-$T$ globally controllable.
\end{proof}

\subsection{$\mathbb{L}^2$-approximate controllability}
In this subsection, we prove $\mathbb{L}^2$-approximate controllability of the system \eqref{M1.N}.

\begin{proof} [Proof of Theorem \ref{L2.app.cont.1}]
Let us fix $\M_0,\M_1\in \mathbb{H}^1(0,2\pi;\mathbb{S}^2)$ and any $\epsilon>0$. We will find a control $\v \in L^2(0,T;\mathbb{L}^2)$ such that the system steer from $\M_0$ to the $\epsilon$-neighbourhood of $\M_1$ in the norm of $\mathbb{L}^2$. Let us consider the Fourier series expansion of $\M_0,\M_1$. For $K\in\Na,$
let us consider the set $$\mathcal{G}_K:=\{(k,l)\,|\,k\leq K,\, l=1,2,3\}\subset \mathbb{N}_0\times \{1,2,3\}.$$  We denote the orthogonal projection $\Pi_K:\mathbb{H}^1\to \mathcal{S}_K,$ where 
$\mathcal{S}_K=\mbox{span}\{\varphi_ie_j\,|\,(i,j)\in \mathcal{G}_K \}$. We note that 
$$\Pi_K(\M_0)\to \M_0 \quad \hbox{and} \quad \Pi_K(\M_1)\to \M_1 \quad \hbox{in}\quad \mathbb{L}^2 \quad \hbox{as} \quad K\to\infty.$$ Hence, we choose $K\in\Na$ large enough such that 
$$|\Pi_K(\M_0)-\M_0|_{\mathbb{L}^2}\leq \epsilon/3 \quad \hbox{and} \quad |\Pi_K(\M_1)-\M_1|_{\mathbb{L}^2}\leq \epsilon/3$$ and set the control modes $\mathcal{K}^1$ with $\mathcal{K}^1\subset\mathcal{G}_K.$
 We coordinatize the elements in the space $S_K$ by $\mathbf{m}$ and its orthogonal complement by $\mathbf{m}^\perp$. By Theorem \ref{glo.cont.1}, we infer that there exists a control $\v_K$
that steers the component $\mathbf{m}$ from $\mathbf{m}(0)=\Pi_K(\M_0)$ to $\mathbf{m}(T_1)=\Pi_K(\M_1)$ in time $T_1>0.$ This control is defined on an interval of length
$T_1,$ which can be chosen arbitrarily small. On the other hand, we note that $\mathbf{m}$ satisfies the system
\begin{align}
\mathbf{m}_t(t,x) & =\mu_1 \mathbf{m}(t,x) \times \mathbf{m}_{xx}(t,x) -\mu_2 \mathbf{m}(t,x) \times (\mathbf{m}(t,x) \times \mathbf{m}_{xx}(t,x))+ \mathbf{m}(t,x) \times\v(t,x),  \notag\\
& \qquad \qquad\qquad\qquad\qquad\qquad\qquad\qquad \qquad\qquad\qquad t\in (0,T],\,\, x\in (0,2\pi), \label{M1.N.m} \\
\mathbf{m}_x(t,0) & =0=\mathbf{m}_x(t,2\pi),\,\, t\in (0,T], \label{M1.N.1.m} \\
\mathbf{m}(0,x) & =\mathbf{m}_0(x),\,\,x \in(0,2\pi) \label{M1.N.2.m},
\end{align}
and  $\mathbf{m}^\perp$ satisfies the system
\begin{align}
\mathbf{m}^\perp_t(t,x) & =\mu_1 \left\{\mathbf{m}(t,x) \times \mathbf{m}^\perp_{xx}(t,x)+\mathbf{m}^\perp(t,x) \times \mathbf{m}_{xx}(t,x)+\mathbf{m}^\perp(t,x) \times \mathbf{m}^\perp_{xx}(t,x)\right\}\notag\\
&\quad -\mu_2 \big\{\mathbf{m}(t,x) \times (\mathbf{m}(t,x) \times \mathbf{m}^\perp_{xx}(t,x))+\mathbf{m}(t,x) \times (\mathbf{m}^\perp(t,x) \times \mathbf{m}_{xx}(t,x))  \notag\\
& \quad +\mathbf{m}(t,x) \times (\mathbf{m}^\perp(t,x) \times \mathbf{m}^\perp_{xx}(t,x))+\mathbf{m}^\perp(t,x) \times (\mathbf{m}(t,x) \times \mathbf{m}_{xx}(t,x)) \notag \\
& \quad + \mathbf{m}^\perp(t,x) \times (\mathbf{m}(t,x) \times \mathbf{m}^\perp_{xx}(t,x)) +\mathbf{m}^\perp(t,x) \times (\mathbf{m}^\perp(t,x) \times \mathbf{m}_{xx}(t,x))\notag \\
&\quad +\mathbf{m}^\perp(t,x) \times (\mathbf{m}^\perp(t,x) \times \mathbf{m}^\perp_{xx}(t,x)\big\},\,\,\,\,
 t\in (0,T],\,\,\,\, x\in (0,2\pi), \label{M1.N.mperp} \\
\mathbf{m}^\perp_x(t,0) & =0=\mathbf{m}^\perp_x(t,2\pi),\,\, \,\,t\in (0,T], \label{M1.N.1.mperp} \\
\mathbf{m}^\perp(0,x) & =\mathbf{m}^\perp_0(x):=(I-\Pi_N)\M_0,\,\,\,\, x \in(0,2\pi) \label{M1.N.2.mperp}.
\end{align}
Using  the vector product identities $\lb a\times b,c\rb_{\Rr^3}=\lb b,c\times a\rb_{\Rr^3}$ and $\lb a\times (b\times c),d\rb_{\Rr^3}=\lb c,(d\times a)\times b\rb_{\Rr^3}$, integration by parts, the Neumann boundary conditions, the structural condition \eqref{eqn-m-saturations} and  the estimate \eqref{exist.mw},  we note that
\begin{align*}
\frac{1}{2}\frac{d}{dt}|\mathbf{m}^\perp(t)|_{\mathbb{L}^2}^2 & =
\mu_1\int_0^{2\pi}\langle \mathbf{m}(t,x) \times \mathbf{m}^\perp_{xx}(t,x),  \mathbf{m}^\perp(t,x)\rangle_{\Rr^3}\,dx  \\
&\quad - \mu_2\int_0^{2\pi}\langle \mathbf{m}(t,x)\times(\mathbf{m}(t,x) \times \mathbf{m}^\perp_{xx}(t,x)),  \mathbf{m}^\perp(t,x)\rangle_{\Rr^3}\,dx \\
&\quad - \mu_2\int_0^{2\pi}\langle \mathbf{m}(t,x)\times(\mathbf{m}^\perp(t,x) \times \mathbf{m}_{xx}(t,x)),  \mathbf{m}^\perp(t,x)\rangle_{\Rr^3}\,dx \\
&\quad - \mu_2\int_0^{2\pi}\langle \mathbf{m}(t,x)\times(\mathbf{m}^\perp(t,x) \times \mathbf{m}^\perp_{xx}(t,x)),  \mathbf{m}^\perp(t,x)\rangle_{\Rr^3}\,dx ,\\
& \leq (\mu_1+3\mu_2)\int_0^{2\pi}|\mathbf{m}^\perp_{x}(t,x)|^2\, dx 
= (\mu_1+3\mu_2)|\mathbf{m}^\perp_{x}(t)|_{\mathbb{L}^2}^2 \leq C.
\end{align*}
  This yields $|\mathbf{m}^\perp(T_1)|_{\mathbb{L}^2}^2\leq |\mathbf{m}^\perp(0)|_{\mathbb{L}^2}^2 + CT_1.$ Hence, for $T_1$ small we have 
$$|\mathbf{m}^\perp(T_1)|_{\mathbb{L}^2}
\leq |(I-\Pi_K)(\M_0)|_{\mathbb{L}^2} + \epsilon/3 \leq \epsilon/3+\epsilon/3 = 2\epsilon / 3 .$$
Therefore, 
\begin{align*}
|\M(T_1)-M_1 |_{\mathbb{L}^2} &= |\mathbf{m}(T_1)+\mathbf{m}^\perp(T_1)-\Pi_K(M_1)+(I-\Pi_K)(\M_1)|_{\mathbb{L}^2} ,\\
  & \leq |\mathbf{m}^\perp(T_1)|_{\mathbb{L}^2} + |(I-\Pi_K)(\M_1)|_{\mathbb{L}^2}, \\
  & \leq 2\epsilon/ 3 + \epsilon/3 = \epsilon.
\end{align*}
This completes the proof.
\end{proof}

\section{support theorems of stochastic LLGEs}
In this section, we prove the support theorems, Theorem \ref{s.glo.cont.1} and 
Theorem \ref{s.L2.app.cont.1}. In the spirit of Theorem 2.8 and Theorem 2.9 in \cite{Erika+Paul}, 
we employ Girsanov Theorem (see Appendix \ref{gir.1}) to prove this result. In \cite{Erika+Paul}, the authors considered the linear equation but in our case, it is a non-linear(quasi-linear) system.  
For $K\in\Na_0$, let us consider the set $$
 \mathcal{G}_K:=\{(k,l)\,|\,k\leq K,\, l=1,2,3\}\subset \mathbb{N}_0\times \{1,2,3\}
\quad{and} \quad  
S_K:=\mbox{span}\{\varphi_ie_j|\,
(i,j)\in\mathcal{G}_K\}\subset \Hh.$$ 
Hence, the resulting $\mathcal{G}_K$-Galerkin stochastic system can be written as (component wise), for $t\in (0,T).$
\begin{align}\label{s.inf.sys.g1}
d m_i^1(t) & =\sum_{\substack{n+k=i,\\ |n-k|=i}}\lambda_k \left(m_n^2(t)m_k^3(t)-m_n^3(t)m_k^2(t)\right)\,dt\notag\\
& \quad - \sum_{\substack{n+k+l=i,\\ |n-k+l|=i,\\ |n+k-l|=i,\\|k+l-n|=i}}
\lambda_k \Big\{m_l^2(t)\left(m_n^1(t)m_k^2(t)-m_n^2(t)m_k^1(t)\right)-m_l^3(t)\left(m_n^3(t)m_k^1(t)-m_n^1(t)m_k^3(t)\right)\Big\}\,dt\notag \\
& \quad + \sum_{j=1}^3\sum_{\substack{n+k=i,\\|n-k|=i\\(k,l)\in \mathcal{K}^1}} m_n^j(t)\langle e_j\times e_l,e_1\rangle d\beta_k^l(t),
\end{align}
\begin{align}\label{s.inf.sys.g2}
d m_i^2(t) & = \sum_{\substack{n+k=i,\\|n-k|=i}}\lambda_k \left(m_n^3(t)m_k^1(t)-m_n^1(t)m_k^3(t)\right)\,dt
\notag\\
& \quad - \sum_{\substack{n+k+l=i,\\|n-k+l|=i,\\ |n+k-l|=i,\\|k+l-n|=i}}
\lambda_k \Big\{m_l^3(t)\left(m_n^2(t)m_k^3(t)-m_n^3(t)m_k^2(t)\right)-m_l^1(t) \left(m_n^1(t)m_k^2(t)-m_n^2(t)m_k^1(t)\right)\Big\}\,dt\notag \\
& \quad + \sum_{j=1}^3\sum_{\substack{n+k=i,\\|n-k|=i\\(k,l)\in \mathcal{K}^1}}m_n^j(t)\langle e_j\times e_l,e_2\rangle d\beta_k^l(t),
\end{align}
and
\begin{align}\label{s.inf.sys.g3}
d m_i^3(t) & = \sum_{\substack{n+k=i,\\|n-k|=i}}\lambda_k \left(m_n^1(t)m_k^2(t)-m_n^2(t)m_k^1(t)\right)\,dt
\notag\\
& \quad - \sum_{\substack{n+k+l=i,\\|n-k+l|=i,\\ |n+k-l|=i,\\|k+l-n|=i}}
\lambda_k  \left\{m_l^1(t)\left(m_n^3(t)m_k^1(t)-m_n^1(t)m_k^3(t)\right) - m_l^2(t) \left(m_n^2(t)m_k^3(t)-m_n^3(t)m_k^2(t)\right)\right\}\,dt\notag \\
& \quad + \sum_{j=1}^3\sum_{\substack{n+k=i,\\|n-k|=i\\(k,l)\in \mathcal{K}^1}} m_n^j(t)\langle e_j\times e_l,e_3\rangle d\beta_k^l(t),
\end{align}
for $(i,j),(k,l),(n,j)\in\mathcal{G}_K.$

\begin{proof}[Proof of Theorem \ref{s.glo.cont.1}]
 Let $T > 0$, $K\in\Na_0$, $R$ and $\eps$ be given as in hypothesis. Let $\mathbf{m}_0, \mathbf{m}_1\in S_K$. We want to find a $\delta>0$ such that there holds
$$\mathbb{P}\Big(|\mathbf{m}(T,\widetilde{\mathbf{m}}_0)-\mathbf{m}_1|_K \leq \eps\Big)  >\delta, $$
for any $\widetilde{\mathbf{m}}_0 \in 
\Big\{ \mathbf{m}: |\mathbf{m}- \mathbf{m}_0|_K \leq R \Big\}
$.

Let us denote the solution of the stochastic system \eqref{s.inf.sys.g1}-\eqref{s.inf.sys.g3} by $\mathbf{m}^s$ and solution of the controlled deterministic system \eqref{inf.sys.g1}-\eqref{inf.sys.g3} by $\mathbf{m}^c.$ By Theorem \ref{glo.cont.1}, the finite dimensional system \eqref{inf.sys.g1}-\eqref{inf.sys.g3} is exactly controllable. Therefore, there exists a control $\v\in \mathbb{L}^2$ such that
\begin{align}\label{cont.6}
\mathbf{m}^c(0)=\mathbf{m}_0 \quad \hbox{and} \quad
\mathbf{m}^c(T)=\mathbf{m}_1.
\end{align}
 Let $\mathbb{Q}$ be the probability measure obtained by using the Girsanov transformation with $$\mathcal{Q}(T)=\frac{d\,\mathbb{Q}}{d\,\mathbb{P}},$$
 where
$$
\mathcal{Q}(T)=e^{- \sum_{(k,l)\in\mathcal{K}^1}\int_0^T v_k^l(t) d\beta_{k}^l(t)-\frac{1}{2}\sum_{(k,l)\in\mathcal{K}^1}\int_0^T |v_k^l(s)|^2\,ds}.
$$ 
By Girsanov Theorem, we can say that for each $(k,l) 
\in \mathcal{K}^1$, 
$$\beta_{k,\v}^l(t):=\beta_{k}^l(t)+ \int_0^t v_k^l(s)\,ds $$ is the Brownian motion with respect to the new measure $\mathbb{Q}.$ Let $\mathbf{m}^s_v$ be the solution of \eqref{s.inf.sys.g1}-\eqref{s.inf.sys.g3} where the Brownian motion $\beta_k^l(t)$ is replaced with $\beta_{k,\v}^l(t)$.
We have
\begin{align}\label{prob1.1}
\mathbb{P}\left(|\mathbf{m}^s(T)-\mathbf{m}_1|_{\mathbb{L}^2} \leq \eps \right)
=\mathbb{Q}\left(|\mathbf{m}^s_v(T)-\mathbf{m}_1|_{\mathbb{L}^2}\leq \eps \right)
& =\mathbb{E}^{\mathbb{Q}}\left[1_{\{|\mathbf{m}^s_v(T)-\mathbf{m}_1|_{\mathbb{L}^2}\leq \eps \}}\right]\notag\\
&=\mathbb{E}^{\mathbb{P}}\left[\mathcal{Q}(T)1_{\{|\mathbf{m}_\v^s(T)-\mathbf{m}_1|_{\mathbb{L}^2} \leq \eps \}}\right].
\end{align}
Therefore, to prove the result, we need to estimate the term  $\mathbb{E}^{\mathbb{P}}\left[\mathcal{Q}(T)1_{\{|\mathbf{m}_\v^s(T)-\mathbf{m}_1|_{\mathbb{L}^2} \leq \eps \}}\right].$
Let $U$ and $V$ be any two non-negative random variables with $ \mathbb{E}^{\mathbb{P}} \left[V^{-1} \right]<\infty$. Then, using H\"{o}lder's inequality we get
\begin{align}\label{eq pq0}
\mathbb{E}^{\mathbb{P}} \left[
(UV)^{\frac{1}{2}} V^{-\frac{1}{2}}  \right]
\leq \left( \mathbb{E}^{\mathbb{P}} \left[UV \right] \right)^{\frac{1}{2}}
\left( \mathbb{E}^{\mathbb{P}} \left[V^{-1} \right] \right)^{\frac{1}{2}},
\end{align}
which yields 
\begin{align}
\mathbb{E}^{\mathbb{P}} \left[
(UV)^{\frac{1}{2}} V^{-\frac{1}{2}}  \right]^2
\leq \left( \mathbb{E}^{\mathbb{P}} \left[UV \right] \right)
\left( \mathbb{E}^{\mathbb{P}} \left[V^{-1} \right] \right).
\end{align}
This implies 
\begin{align}\label{eq pq}
\mathbb{E}^{\mathbb{P}} \left[UV \right]
\geq \dfrac{\left( \mathbb{E}^{\mathbb{P}} \left[(UV )^{\frac{1}{2}} 
 V^{-\frac{1}{2}} \right] \right)^2}{\mathbb{E}^{\mathbb{P}} \left[V^{-1} \right]}
=  \dfrac{\left( \mathbb{E}^{\mathbb{P}} [U ^{\frac{1}{2}} ] \right)^2}{\mathbb{E}^{\mathbb{P}} \left[V^{-1} \right]}
 .
\end{align}
Hence, if we choose $U=1_{\{|\mathbf{m}_\v^s(T)-\mathbf{m}_1|\leq \eps\}}$ and $V=\mathcal{Q}(T)$ in \eqref{eq pq0}, then we have from \eqref{eq pq}
\begin{align}
\mathbb{E}^{\mathbb{P}}\left[1_{\{|\mathbf{m}_\v^s(T)-\mathbf{m}_1|_{\mathbb{L}^2} \leq \eps\}}\right]
&\geq \frac{\left(\mathbb{E}^{\mathbb{P}}\left[1_{\{|\mathbf{m}_\v^s(T)-\mathbf{m}_1|_{\mathbb{L}^2} \leq \eps\}}\right]\right)^2}{\mathbb{E}^{\mathbb{P}}\left[\left|\mathcal{Q}^{-1}(T)\right|\right]}.
\end{align}
It is easy to show that there exists $C>0$ such that
\begin{align}\label{prob2}
\mathbb{E}^{\mathbb{P}}\left[\left|\mathcal{Q}^{-1}(T)\right|\right]\leq C e^{C\sum_{(k,l)\in\mathcal{K}^1}\int_0^T |v_k^l(s)|^2\,ds}.
\end{align}
Indeed, the process  $\mathcal{Q}^{-1}(t)$ satisfies the stochastic differential equation
\begin{align}
\left\{\begin{array}{ll}
d \mathcal{Q}^{-1}(t)& = -\mathcal{Q}^{-1}(t)\sum_{(k,l)\in\mathcal{K}^1} v_k^l(t) d\beta_{k}^l(t)+ \mathcal{Q}^{-1}(t)\sum_{(k,l)\in\mathcal{K}^1} |v_k^l(t)|^2\,dt,\quad t\in (0,T), \\
 \mathcal{Q}^{-1}(0)& =1.
\end{array}\right.
\end{align}
Hence, for $t\in (0,T)$
\begin{align*}
 \mathcal{Q}^{-1}(t)& = 1+ \int_0^t\mathcal{Q}^{-1}(s)\sum_{(k,l)\in\mathcal{K}^1} v_k^l(s) d\beta_{k}^l(s)+ \int_0^t\mathcal{Q}^{-1}(s)\sum_{(k,l)\in\mathcal{K}^1} |v_k^l(s)|^2\,ds.
\end{align*}
We now take supremum followed by expectation raise to the power $p$. Then using the Burkholder-Davis-Gundy inequality and the Young inequality, we achieve
\begin{align*}
& \mathbb{E}^{\mathbb{P}}\Big[\sup_{0\leq t\leq T}\Big|\mathcal{Q}^{-1}(t)\Big|\Big]
 \\
 & \leq  1+ \mathbb{E}^{\mathbb{P}}\Big[\sup_{0\leq t\leq T}
 \Big|\int_0^t\mathcal{Q}^{-1}(s)\sum_{(k,l)\in\mathcal{K}^1} v_k^l(s) d\beta_{k}^l(s)\Big|\Big]
 + \mathbb{E}^{\mathbb{P}} \Big[\int_0^T\mathcal{Q}^{-1}(t)\sum_{(k,l)\in\mathcal{K}^1} |v_k^l(t)|^2\,dt\Big] \\
 & \leq  1+C \mathbb{E}^{\mathbb{P}}\Big[\int_0^T (\mathcal{Q}^{-1})^2(t)\sum_{(k,l)\in\mathcal{K}^1} |v_k^l(t)|^2\,dt \Big]^\frac{1}{2}
 + \mathbb{E}^{\mathbb{P}} \Big[\int_0^T\mathcal{Q}^{-1}(t)\sum_{(k,l)\in\mathcal{K}^1} |v_k^l(t)|^2\,dt \Big] \\
 & \leq  1+C \mathbb{E}^{\mathbb{P}}\Big[\Big(\sup_{0\leq t\leq T}\mathcal{Q}^{-1}(t)\Big)^\frac{1}{2}\Big(\int_0^T \mathcal{Q}^{-1}(t)\sum_{(k,l)\in\mathcal{K}^1} |v_k^l(t)|^2\,dt\Big)^\frac{1}{2} \Big] \\
 &\qquad\qquad+  \mathbb{E}^{\mathbb{P}} \Big[\int_0^T\mathcal{Q}^{-1}(t)\sum_{(k,l)\in\mathcal{K}^1} |v_k^l(t)|^2\,dt \Big] \\
 & \leq  1+ \frac{1}{2}\mathbb{E}^{\mathbb{P}}\Big[\sup_{0\leq t\leq T}\Big|\mathcal{Q}^{-1}(t)\Big|\Big]
 + C \mathbb{E}^{\mathbb{P}} \Big[\int_0^T\mathcal{Q}^{-1}(t)\sum_{(k,l)\in\mathcal{K}^1} |v_k^l(t)|^2\,dt \Big]. \end{align*}
 This yields 
 \begin{align*}
&\frac{1}{2} \mathbb{E}^{\mathbb{P}}\Big[\sup_{0\leq t\leq T}\Big|\mathcal{Q}^{-1}(t)\Big|\Big]
 \leq  1 + C \mathbb{E}^{\mathbb{P}} \Big[\int_0^T\mathcal{Q}^{-1}(t)\sum_{(k,l)\in\mathcal{K}^1} |v_k^l(t)|^2\,dt \Big].
\end{align*}
Finally, the Gronwall inequality yields us
\begin{align*}
 \mathbb{E}^{\mathbb{P}}\Big[\sup_{0\leq t\leq T}\Big|\mathcal{Q}^{-1}(t)\Big|\Big]
 \leq C e^{C\sum_{(k,l)\in\mathcal{K}^1}\int_0^T |v_k^l(s)|^2\,ds}.
\end{align*}
Using \eqref{cont.6}, we obtain
\begin{align*}
\mathbb{E}^{\mathbb{P}}\left[1_{\{|\mathbf{m}^s_v(T)-\mathbf{m}_1|_{\mathbb{L}^2} \leq \eps\}}\right]& = \mathbb{E}^{\mathbb{P}}\left[1_{\{|\mathbf{m}_\v^s(T)-\mathbf{m}^c(T)|_{\mathbb{L}^2}\leq \eps \}}\right] .
\end{align*}
We recall that $\mathbf{m}^s$ satisfies
\begin{align*}
d\,\mathbf{m}^s_v(t) &= \Big\{\mathbf{m}^s_v(t)\times (\mathbf{m}^s_v)_{xx}(t)-\mathbf{m}^s_v(t)\times \Big(\mathbf{m}^s_v(t)\times (\mathbf{m}^s_v)_{xx}(t)\Big)\Big\}\, dt + \mathbf{m}^s_v(t)\times d\,\pmb{\beta}_v(t)\\
& = \Big\{(\mathbf{m}^s_v)_{xx}(t)+\mathbf{m}^s_v(t)\times (\mathbf{m}^s_v)_{xx}(t)
+\mathbf{m}^s_v(t)\big|(\mathbf{m}^s_v)_{x}(t)\big|_{\Rr^3}^2 \Big\}\, dt+ \mathbf{m}^s_v(t)\times d\,\pmb{\beta}_v(t)
\end{align*}
where $\pmb{\beta}_{\v}(t)=\sum_{(k,l)\in\mathcal{K}^1}\varphi_ke_l\beta_{k,\v}^l(t)$
and
$\mathbf{m}^c$ satisfies
\begin{align*}
d\,\mathbf{m}^c(t)= \Big\{\mathbf{m}^c_{xx}(t)+\mathbf{m}^c(t)\times \mathbf{m}^c_{xx}(t)+\mathbf{m}^c(t)|\mathbf{m}^c_{x}(t)|_{\Rr^3}^2 \Big\}dt+ \mathbf{m}^c(t)\times \v(t)\, dt.
\end{align*}
Hence, $\u=\mathbf{m}^s_v-\mathbf{m}^c$ satisfies
\begin{align}\label{diffsc.1}
d\u(t)&=\Big\{\u_{xx}(t)+\mathbf{m}^s_v(t)\times \u_{xx}(t)+\u(t)\times \mathbf{m}^c_{xx}(t) \notag \\
& \quad + \u(t)|(\mathbf{m}_v^s)_x(t)|_{\Rr^3}^2 + \mathbf{m}^c(t)\Big(|(\mathbf{m}^s_v)_{x}(t)|_{\Rr^3}^2-|\mathbf{m}^c_{x}(t)|_{\Rr^3}^2 \Big)\Big\}\,dt \notag \\
&\quad + \u(t)\times \v(t)\,dt + \mathbf{m}^s(t) \times \,d\pmb{\beta}(t)
\end{align}
where $\pmb{\beta}_{\v}(t)=\pmb{\beta}(t)+\int_0^t\v(s)\,ds$. Applying the It\^{o} Lemma to the
function $\phi(x)=\dfrac{1}{2}|x|_{\mathbb{L}^2}^2$ and to the process $u(t),\,\,t \in [0,T]$, we obtain
\begin{align*}
\frac{1}{2} |u(T)|_{\mathbb{L}^2}^2 -\frac{1}{2} |u(0)|_{\mathbb{L}^2}^2
& =\int_0^T\int_0^{2\pi}
\Big\{-|u_x(t)|_{\Rr^3}^2 
 - \langle\mathbf{m}^s(t)\times\u_{xx}(t)),\u(t) \rangle_{\Rr^3} \\
 & \quad +|\u(t)|_{\Rr^3}^2|\mathbf{m}_x^s(t)|_{\Rr^3}^2 
 + \Big\langle\mathbf{m}^c(t)\left(|\mathbf{m}^s_{x}(t)|_{\Rr^3}^2-|\mathbf{m}^c_{x}(t)|_{\Rr^3}^2 \right),\u(t) \Big\rangle_{\Rr^3}
\Big\}\,dx\,dt \\
& \quad + \sum_{(k,l)\in\mathcal{K}^1}\int_0^T\int_0^{2\pi} \big\langle \mathbf{m}^s(t) \times \varphi_k(x)e_l,\u(t) \big\rangle_{\Rr^3} dx \,d\beta_{k,\v}^l(t)\\
& \quad +\frac{1}{2} \sum_{(k,l)\in\mathcal{K}^1}\int_0^T\int_0^{2\pi}|\mathbf{m}^s(t) \times \varphi_k(x)e_l|_{\Rr^3}^2\,dx\,dt.
\end{align*}
Using $u(0)=0$ and proceeding as in the proof of Lemma 5.2 of \cite{BMM}, we obtain
\begin{align}\label{gron P}
\mathbb{E}^{\mathbb{P}}|\u(T)|_{\mathbb{L}^2}^2 
& \leq   C \Big(\mathbb{E}^{\mathbb{P}}\int _0^T \left|\pmb{\beta}(t)\right|\, dt\Big) \, e^{\int_0^T\phi_C(s)ds}
\end{align}
for some integrable function $\phi_C$ depending on the spatial derivatives of $\mathbf{m}_v^s$ and $\mathbf{m}^c$.
By Chebyshev's inequality we have
\begin{align*}
\mathbb{P}\left(|\u(T)|_{\mathbb{L}^2}\geq \eps \right)\leq 
\frac{1}{\eps^2}\mathbb{E}^{\mathbb{P}}|\u(T)|_{\mathbb{L}^2}^2 .
\end{align*}
Consequently, we obtain for any $\epsilon>0$ large enough 
\begin{align}\label{prob3}
\mathbb{P}\left(|\u(T)|_{\mathbb{L}^2}\leq \eps \right)
\geq 1-\frac{1}{\eps^2}\mathbb{E}^{\mathbb{P}}|\u(T)|_{\mathbb{L}^2}^2>0.
\end{align}
Finally, using \eqref{prob2} and \eqref{prob3} we get from \eqref{prob1}  
$$
\mathbb{P}\left(|\mathbf{m}^s(T)-\mathbf{m}_1 |_{\mathbb{L}^2}\leq \eps\right) 
 >0.
$$
This completes the proof.
\end{proof}

\begin{proof}[Proof of Theorem \ref{s.L2.app.cont.1}]
Let $\mathbf{M}^s(\cdot)$ be the solution of the system \eqref{sM1.N} with $\mathbf{M}^s(0)=\mathbf{M}_0$. Since the system \eqref{M1.N} is approximately controllable, there exists a control $\v \in L^2(0,T;\mathbb{L}^2)$ of the form \eqref{control.1} such that the solution $\mathbf{M}^c$ of the system \eqref{M1.N} satisfies 
\begin{align}
|\mathbf{M}^c(T)- \mathbf{M}_1|_{\mathbb{L}^2} \leq \dfrac{\eps}{3}.
\end{align}
Proceeding as in the proof of Theorem \ref{s.glo.cont.1}, we consider the Girsanov transformation
 $$\mathcal{Q}(T)=\frac{d\,\mathbb{Q}}{d\,\mathbb{P}},$$
 where
$$
\mathcal{Q}(T)=e^{- \sum_{(k,l)\in\mathcal{K}^1}\int_0^T v_k^l(t) d\beta_{k}^l(t)-\frac{1}{2}\sum_{(k,l)\in\mathcal{K}^1}\int_0^T |v_k^l(s)|^2\,ds}.
$$ 
By Girsanov Theorem, we can say that for each $(k,l) 
\in \mathcal{K}^1$, 
$$\beta_{k,\v}^l(t):=\beta_{k}^l(t)+ \int_0^t v_k^l(s)\,ds $$ is the Brownian motion with respect to the new measure $\mathbb{Q}.$ Let $\mathbf{M}^s_v$ be the solution of \eqref{M1.N} where the Brownian motion $\beta_k^l(t)$ is replaced with $\beta_{k,\v}^l(t)$. As in the proof of Theorem \ref{s.glo.cont.1}, we will find
\begin{align}\label{prob1}
\mathbb{P}\left(|\mathbf{M}^s(T)-\mathbf{M}_1|_{\mathbb{L}^2} \leq \eps \right)
&\geq \frac{\left(\mathbb{E}^{\mathbb{P}}\left[1_{\{|\mathbf{m}_\v^s(T)-\mathbf{m}_1|_{\mathbb{L}^2} \leq \eps\}}\right]\right)^2}{\mathbb{E}^{\mathbb{P}}\left[\left|\mathcal{Q}^{-1}(T)\right|\right]}.
\end{align} 

\end{proof}
\begin{appendix}

\section{Geometric control theory and controllability of non-linear systems}\label{app geo cont}
In view of \cite{BoSi}, we present some basic facts about the controllability of non-linear finite dimensional systems. We introduce the concepts of Lie bracket and state a sufficient condition which guarantees the controllability of the the system.
For $N,r\in\mathbb{N}$, let us consider the control system in $\mathbb{R}^N,$
\begin{align}\label{controls.1}
\dot{u}(t)& =F(u(t),v(t)),\quad t>0;
\end{align}
\begin{itemize}
\item $u:[0,\infty)\to \mathbb{R}^N$ is the state of the system;
\item $V\subset \mathbb{R}^r$ is the set of control values and $v :[0, \infty)\to V$ is a control;
\item $F$ is a smooth function of its arguments and that $u$ is regular enough in such a way that equation \eqref{controls.1} with the initial condition $u(0) = u_0\in \mathbb{R}^N$ has an
unique solution in $[0,\infty)$.
\end{itemize}
Let $t\in [0,\infty)\mapsto u(t; u_0, v)$ be the solution of \eqref{controls.1} starting from $u_0 \in \mathbb{R}^N$ at $t = 0$ and corresponding to a control function $u.$
We recall the following definitions:
\begin{defi}
Let $u_0\in\mathbb{R}^N.$
\begin{itemize}
\item (Reachable set from $u_0$ at time $T\geq 0$): 
$$
\mathcal{A}(T, u_0 ) = \big\{u_1 \in \mathbb{R}^N\,| \,\,\exists\, v: [0, T ] \to V,\, u\left(T; u_0 , u\right) = u_1 \big\}.
$$
\end{itemize}
\end{defi}
\begin{defi}[Controllable]
The system \eqref{controls.1} is said to be
 controllable in time $T$ if for every $u_0 \in \mathbb{R}^N,\, \mathcal{A}(T,u_0 ) = \mathbb R^N$.
\end{defi}
\begin{defi}
A vector space $X$ over field $\mathbb{R}$ endowed with an operation $[\cdot,\cdot]: X \times X \to X$
is said to be a Lie algebra if it satisfies the following
\begin{itemize}
\item (bilinear): for every $\lambda_1 , \lambda_2 \in \mathbb{R},$ $f,g,f_1,f_2,g_1,g_2 \in X,$
\begin{align*}
[f, \lambda_1 g_1 + \lambda_2 g_2 ] = \lambda_1 [f, g_1 ] + \lambda_2 [f, g_2 ], \\
[\lambda_1 f_1 + \lambda_2 f_2 , g] = \lambda_1 [f_1 , g] + \lambda_2 [f_2 , g];
\end{align*}
\item (antisymmetric): for every $f,g \in X,$
$[f,g]=-[g,f];
$
\item (Jacobi identity): for every $f,g,h \in X$
$$[f, [g, h]] + [h, [f, g]] + [g, [h, f ]] = 0.$$
\end{itemize}
\end{defi}

Let Vec($\mathbb{R}^N $) be the vector space of all smooth vector fields in $\mathbb{R}^N $. Let us first define the Lie bracket
between two vector fields $f,g\in$Vec($\mathbb{R}^N $), as the vector field defined by
$$
[f, g](x) = Dg(x)f (x) - Df (x)g(x),\quad x\in \mathbb{R}^N.
$$

Here, for a given vector field $f = (f_1 ,\dots, f_N )^t\in \mathbb{R}^N$, $Df$ is the matrix of partial derivatives
of the components of $f$ given by
$$
\begin{pmatrix}
\partial_1 f_1 & \cdots & \partial_N f_1 \\
\cdot & \cdots & \cdot \\
\cdot & \cdots & \cdot \\
\partial_1 f_N & \cdots & \partial_N f_N
\end{pmatrix}.
$$
We note that that (Vec($\mathbb{R}^N$), [·, ·]) is a Lie algebra.

{\bf Example:} Let  $f(x)=Ax,\,g(x)=Bx,\,x\in\mathbb{R}^N$  for $A,B\in \mathbb{R}^{N\times N}$. 

Then, $f,g \in$Vec($\mathbb{R}^N)$ and $[f,g](x)=ABx-BAx,\,x\in\mathbb{R}^N$.
\begin{defi}
Let $\mathcal{F}$ be a family of vector fields in Vec($\mathbb{R}^N$) . We call Lie($\mathcal{F}$) the smallest 
sub-algebra of Vec($\mathbb{R}^N$) containing $\mathcal{F}$.

In particular, Lie($\mathcal{F}$) is the span of all vector fields of $\mathcal{F}$ and of their iterated Lie brackets of any order:
$$
   \mbox{Lie}(\mathcal{F})=\mbox{span}\big\{f_1 , [f_1 , f_2 ], [f_1 , [f_2 , f_3 ]], [f_1 , [f_2 , [f_3 , f_4 ]]], \dots, |\,\, f_1 , f_2 ,\dots\in\mathcal{F}\big\}.
$$
\end{defi}

\begin{defi}
The family $\mathcal{F}$ is said to be Lie bracket generating at a point $x\in\mathbb{R}^N$ if the
dimension of Lie$_x(\mathcal{F}):= \{f (x) |\, f \in \mbox{Lie}(\mathcal{F})\}$ is equal to $N.$

The family $\mathcal{F}$ is said to be Lie bracket generating if this condition is verified for every $x\in\mathbb{R}^N$.
\end{defi}

We note that in general Lie($\mathcal{F}$) is an infinite-dimensional subspace of Vec($\mathbb{R}^n$) , while
Lie$_x(\mathcal{F})$ is a subspace of $\mathbb{R}^N$.

\begin{defi}
A family of vector fields $\mathcal{F}$ is said to be symmetric if $f \in\mathcal{F}$ implies  $-f\in \mathcal{F}$.
\end{defi}

 Let us denote the collection of smooth and complete vector fields parametrized by elements in $V\subset \mathbb{R}^r$ by
$$
\mathcal{F}_V:=\big\{F(\cdot, v):\mathbb{R}^N\to \mathbb{R}^N ,\,\,  v\in V\big\}.
$$

\begin{thm}[Chow-Rashevskii, ] \label{ChowR}
If $\mathcal{F}$ is Lie bracket generating and symmetric,
then for every $u_0\in\mathbb{R}^N$ we have $\mathcal{A}^T_{\mathcal{F}_V}(u_0) =\mathbb{R}^N$
\end{thm}
\begin{defi}[Extension of control system]
A family $\mathcal{F}'$ of real analytic vector fields in $\RN$ is said to be an
extension of $\mathcal{F}$ if $\mathcal{F}\subset \mathcal{F}'$ and for all 
$u_0\in \RN,$ $\overline{\mathcal{A}^T_{\mathcal{F}}(u_0)}=\overline{\mathcal{A}^T_{\mathcal{F}'}(u_0)},$ where the over line denotes the closure of the set in $\RN.$ 
\end{defi}
\begin{defi}[Fixed time extension of control system]
A family $\mathcal{F}'$ of real analytic vector fields in $\RN$ is said to be a fixed time extension of $\mathcal{F}$ if $\mathcal{F}\subset \mathcal{F}'$ and for all $T>0$ and for all $u_0\in \RN,$ $\overline{\mathcal{A}_{\mathcal{F}}^T(u_0)}=\overline{\mathcal{A}_{\mathcal{F}'}^T(u_0)},$ where the over line denotes the closure of the set in $\RN.$ 
\end{defi}
\begin{lem}[Agrachev-Sarychev \cite{Agra_Sary.3}]
Let $\mathcal{F}'$ be an extension of $\mathcal{F}.$ If  $\mathcal{F}'$ is globally controllable in time $T$ from $u_0\in\RN,$ then $\overline{\mathcal{A}^T_{\mathcal{F}}(u_0)}=\RN.$
\end{lem}

\begin{defi}(Control affine system)\cite[Section 13.2.3]{LaValle}
Let $N,r\in\Na,\, r\leq N.$ Let $u:[0,T]\to \RN$ be the state of the system, $v_j:[0,T]\to \Rr$ for $j=1,2,\dots,r$ be controls and $f,g_j$ for $j=1,2,\dots,r$ be
smooth vector fields in $\RN$. A non-linear control system of the form
\begin{align}\label{control.affine.d}
\dot{u}(t)&=f(u(t))+\sum_{j=1}^r g_j(u(t))v_j(t),\,\,t\in (0,T]\quad \hbox{with} \quad
 u(0)= u_0,
\end{align}
is called a control-affine system or affine-in-control system.
\end{defi}
Affine means `linear', and non-affine means `nonlinear'. Therefore, a nonlinear system in which the control appears linearly is called a control-affine nonlinear system (or simply control-affine system, where the nonlinearity with respect to the state is already there in the system), whereas a system that has nonlinearities both in the state and in the controls is called a control non affine nonlinear system (or simply control-non affine system).
The general form of an control-affine system is represented  as :
$$
\dot{u}(t) = f(u(t)) + g(u(t))v(t),\,\,t\in (0,T]
$$
whereas the general form of a system which is non affine in the controls is given as:
$$
\dot{u}(t) = f(u(t)) + g(u(t),v(t)),\,\,t\in (0,T].
$$
\subsection{Bracket generating, approximate controllability and global controllability}
Let us come back to the control affine system \eqref{control.affine.d} again.
\begin{defi} [Full Lie rank property] \label{liegen}
The control affine system \eqref{control.affine.d} is said to have the full Lie rank property (or Lie bracket generating) if for every $u_0\in\RN$ the iterated Lie brackets of the vector fields $f,g_j$ for $j=1,2,\dots, r$ evaluated at $u_0\in\RN$ span the whole space $\RN$.
\end{defi}
\begin{defi}[Normally Accessible,  Jurdjevic \cite{Jurd}, Chapter 3 ]
A point $u_0\in\RN$ is normally accessible from another point $u_1\in\RN$ by a family of vector fields $\mathcal{F}$ of $\RN$ if there exist elements $f_i\in\mathcal{F}$ for $i=1,\dots,p$ such that the function $F:\Rr^p\to \RN$ defined by
$$
F(t):=\exp (t_1f_1)\circ \exp (t_2f_2)\circ \cdots\circ  \exp (t_pf_p)(u_1),\quad t=(t_1,t_2,\dots,t_p)\in \Rr^p,
$$ 
satisfies the following:
\begin{enumerate}[i.]
\item there exists $\hat{t}\in \Rr^p$ such that $F(\hat{t})=u_0$,
\item rank of $DF(\hat{t})=N.$
\end{enumerate}
We shall say that $u_0$ is normally accessible from  $u_1$ in time 
$\sum_{i=1}^p\hat{t}_i.$
\end{defi}
If  $u_0\in\RN$ is normally accessible from $u_1\in\RN,$ then by Constant Rank Theorem, there exists a neighbourhood $U$ of $\hat{t}$ in $\Rr^p$ such that $F(U)$ is a neighbourhood $u_0$ in $\RN.$ Therefore, 
any point which is normally accessible from $u_0$ at time $T$ must be an interior of the attainable set $\mathcal{A}_{\mathcal{F}}(u_0).$

 Let us assume that a Galerkin approximated system has approximate controllability property. Now, to show that the Galerkin approximated system has exact controllability property, we use the following result of Agrachev-Sarychev \cite[Theorem 5.5]{Agra_Sary.3}, Jurdjevic \cite[Chapter 3, Theorem 1]{Jurd}.
\begin{thm}[Agrachev-Sarychev \cite{Agra_Sary.3}, Jurdjevic \cite{Jurd}]
Let the control system \eqref{control.affine.d} be
real-analytic. Then, it is accessible if and only if it is bracket generating.
Moreover, if the control system \eqref{control.affine.d} is $C^\infty$-smooth and is bracket generating, then, it is accessible. 
\end{thm}
 For a bracket generating system, the approximate controllability yields the exact controllability. We have 
 \begin{prop}[Jurdevic \cite{Jurd}] \label{cont.1}
 Let the system \eqref{control.affine.d} be bracket generating and 
$\overline{\mathcal{A}_{\mathcal{F}}^T(u_0)}=\RN$. Then $\mathcal{A}_{\mathcal{F}}^T(u_0)=\RN.$
\end{prop}
\section{Stochastic preliminaries}
Since to keep our article self-contained, we present here some tools from stochastic theory which will be used to prove Theorem 1.5 and Theorem 1.6.
\subsection{Girsanov Transformation}\label{gir.1}
In this section we will briefly discuss Girsanov Transformation see Girsanov\cite{Gir}.
See also Kuo \cite[Chapter 8, Page-141]{Kuo}, Da Prato-Zabczyk\cite[Chapter 10, Page-291]{DapratoZ}.

Let us consider a probability space $(\Omega, \mathcal{F}, \mathbb{P})$ together with a normal filtration $\{\mathcal{F}_t\}_{0\leq t\leq T}$. Let  $H$ and $U$ be two Hilbert spaces and $Q$ be a self-adjoint bounded non-negative operator on U. Let $W(t),\, 0\leq t\leq T$ be a $Q$-Wiener process  in another Hilbert space $U_1\supset U$ and $U_0=Q^\frac{1}{2}(U)$
be the Hilbert space with the inner product 
$$
\langle u, v\rangle_0 := \langle Q^{-\frac{1}{2}}u, Q^{-\frac{1}{2}} v\rangle,\quad u,v \in Q^\frac{1}{2}(U),
$$ and norm $|u|_0=(\langle u, u\rangle_0)^\frac{1}{2}$, where $\langle \cdot, \cdot\rangle$ is the inner product in $U$.
We are interested in the following question. 
\\
{\it Question}: Are there other stochastic processes $\varphi(t)$ for $0\leq t\leq T,$ such that 
$W(t)-\varphi(t)$ for $0\leq t\leq T,$ is a Brownian motion with respect to some probability measure ?

We will show the following: 
\newline
 If $$\varphi(t) = \int_0^t \psi(s) ds,\,\, 0 \leq t \leq T,\quad \hbox{with} \quad \psi \in \mathcal{L}_{ad}(\Omega,L^2(0, T;U_0))$$
satisfies the condition in Theorem 8.7.3, then $\widehat{W}(t) = W(t) - \varphi(t)$ is indeed a Wiener process with respect to the probability measure $d\widehat{\mathbb{P}} = \mathcal{E}_\psi(T)d\mathbb{P}$. Here, $\mathcal{E}_\psi(t)$
 is the exponential process given by $\psi$ as defined by 
 $$
\mathcal{E}_\psi(t):=e^{\int_0^t \left\langle \psi(s),dW(s) \right\rangle_0 - \frac{1}{2}\int_0^t|\psi(s)|_0^2\,ds},\quad t\in [0,T].
$$
The main idea in the proof of this important theorem, due to Girsanov \cite{Gir}, is the transformation of probability measures described in the previous section and L\'{e}vy’s characterization theorem of Brownian motion.

The following result is proved in \cite{Bens} and Kozlov \cite{Koz}. We also refer Da Prato-Zabczyk \cite[Page-291]{DapratoZ}.
\begin{thm}\label{Girt}
Let $\psi(t)$ be a $U_0$-valued $\mathcal{F}_t$-predicatble process such that 
\begin{align}\label{NSexp}
\mathbb{E} \left(\mathcal{E}_\psi(T)\right)=1.
\end{align}
Then, the process 
\begin{align}\label{NW}
\widehat{W}(t)=W(t)- \frac{1}{2}\int_0^t\psi(s)\,ds,\quad {t\in [0,T]}
\end{align}
is a $Q$-Wiener process with respect to $\{\mathcal{F}_t\}_{0\leq t\leq T}$
on the probability space $(\Omega, \mathcal {F}, \widehat{\mathbb{P}})$
where
\begin{align}\label{NP}
d\,\widehat{\mathbb{P}}=\mathcal{E}_\psi(T)\, d\,\mathbb{P}.
\end{align}
\end{thm}
\begin{proof}
Let us first assume that $\psi$ is bounded i.e., there exists $K>0$ such that  $|\psi(t)|_0 \leq K$ for all $t\in [0,T].$ Let $g:[0,T]\to U_0$ be bounded Borel measurable. Then, we note from the definition of $\widehat{W}$ that 
$$
\int_0^T \langle g(t), d\,\widehat {W}(t)\rangle_0 = \int_0^T \langle g(t), d\,W(t)\rangle_0 -\int_0^T \langle g(t), \psi(t)\rangle_0 d\,t.
$$
Let us denote the Expectation with respect to $\widehat{\mathbb{P}}$ by
$\widehat{\mathbb{E}}.$ Then, we compute
\begin{align}\label{exp.1}
\widehat{\mathbb{E}}\left(e^{\int_0^T \langle g(t), d\,\widehat {W}(t)\rangle_0}\right)& = \int_\Omega e^{\int_0^T \langle g(t), d\,\widehat {W}(t)\rangle_0} \,d\, \widehat{\mathbb{P}}
                 = \int_\Omega e^{\int_0^T \langle g(t), d\,W(t)\rangle_0 -\int_0^T \langle g(t), \psi(t)\rangle_0 d\,t} \mathcal{E}_\psi(T) \,d\, \mathbb{P} \notag \\
                & = \int_\Omega e^{\int_0^T \langle g(t), d\,W(t)\rangle_0 -\int_0^T \langle g(t), \psi(t)\rangle_0 d\,t + \int_0^T \left\langle \psi(s),dW(s) \right\rangle_0 - \frac{1}{2}\int_0^T|\psi(s)|_0^2\,ds }  \,d\, \mathbb{P},\\
                & = e^{\frac{1}{2}\int_0^T|\psi(s)|_0^2\,ds}\mathbb{E}\left( e^{\int_0^T \langle g(t)+\psi(t), d\,W(t)\rangle_0  - \frac{1}{2}\int_0^T|g(t)+\psi(t)|_0^2\,dt }   \right),\notag\\
                & = e^{\frac{1}{2}\int_0^T|\psi(t)|_0^2\,dt},
\end{align}
where we use that fact that $\mathbb{E}\left( e^{\int_0^T \langle g(t)+\psi(t), d\,W(t)\rangle_0  - \frac{1}{2}\int_0^T|g(t)+\psi(t)|_0^2\,dt }   \right)=1.$  This follows from Lemma \ref{r1} below, and
the fact that the process $\gamma(t)=|g(t)+\psi(t)|_0$ is bounded by a constant, and from the observation that for a bounded process $\gamma(t)$ and a real valued Wiener process $\beta$
$$
\mathbb{E}\left( e^{\int_0^T \gamma(t)d\,\beta (t)  - \frac{1}{2}\int_0^T\gamma(t)^2\,dt }   \right)=1.
$$
We can compute that 
$$
\widehat{\mathbb{E}}\left(e^{\lambda\int_0^T \langle g(t), d\,\widehat {W}(t)\rangle_0}\right) = e^{\frac{\lambda^2}{2}\int_0^T|\psi(t)|_0^2\,dt},
$$
for all $\lambda\in \mathbb{R}.$ Hence, considering 
$$
h(z)= \widehat{\mathbb{E}}\left(e^{z\int_0^T \langle g(t), d\,\widehat {W}(t)\rangle_0}\right),\quad z\in\mathbb{C}
$$
we note that $h$ is finite for all real numbers $z.$ It then follows that $h$ is well defined for all complex numbers and is continuously differentiable with respect to the complex variable $z.$ Therefore $h$ is analytic on $\mathbb{C}$ and
$$
h(z)=  e^{\frac{z^2}{2}\int_0^T|\psi(t)|_0^2\,dt},\quad z\in \mathbb{C}.
$$
Therefore,
$$
\widehat{\mathbb{E}}\left(e^{i\lambda\int_0^T \langle g(t), d\,\widehat {W}(t)\rangle_0}\mathbbm{1}_A\right) = e^{-\frac{\lambda^2}{2}\int_0^T|\psi(t)|_0^2\,dt}\widehat{\mathbb{P}}(A),
$$
for all $\lambda\in \mathbb{R}$ and $A\in \mathcal{F}_0.$ Therefore, Therefore random variables
$\int_0^T \langle g(t), d\,\widehat {W}(t)\rangle_0 $ are Gaussian with covariances $\int_0^T|\psi(t)|_0^2\,dt.$ By the same one can show that
\begin{align}\label{GaIn}
\widehat{\mathbb{E}}\left(e^{i\lambda\int_t^T \langle g(t), d\,\widehat {W}(t)\rangle_0}\mathbbm{1}_A\right) = e^{-\frac{\lambda^2}{2}\int_t^T|\psi(t)|_0^2\,dt}\widehat{\mathbb{P}}(A),
\end{align}
for all $\lambda\in \mathbb{R}$ and $A\in \mathcal{F}_t.$ 
Consequently, the random variables $\int_0^T \langle g(t), d\,\widehat {W}(t)\rangle_0 $ are independent of $\mathcal{F}_t.$ This way the proof of the theorem is complete under the condition that $\psi$ is a
bounded process. For a general process $\psi$ satisfying \eqref{NSexp} consider a sequence $\psi_N$ of bounded processes such that
$$
\lim_{N\to\infty}\int_0^T|\psi(t)-\psi_N(t)|_0^2\,d\,t=0,\quad \mathbb{P} \mbox{ a.s.}
$$
and define processes
$$
\widehat{W}_N(t)=W(t)- \frac{1}{2}\int_0^t\psi_N(s)\,ds,\quad t\in [0,T],\,\,N\in\Na.
$$
Using \eqref{NSexp} and \eqref{GaIn}, then passing to the limit $N\to\infty$
we note that \eqref{GaIn} holds for general case.
\end{proof}

\begin{lem}\label{r1}
Let $\psi(t)$ be a $U_0$-valued $\mathcal{F}_t$-predictable process such that 
\begin{align}\label{NSp}
\mathbb{P} \left(\int_0^T|\psi(t)|_0^2\,dt < \infty \right)=1.
\end{align}
Then, there exists a real valued Wiener process $\beta(t),\, t\in[0, T ],$ with respect to $\{\mathcal{F}_t\}_{0\leq t\leq T}$, normalized and such that $\mathbb{P}$-a.s.
$$
\int_0^t |\psi(s)|_0,d\beta(s)  = \int_0^t \left\langle \psi(s),dW(s) \right\rangle_0,\quad t\in [0,T].
$$
\end{lem}


\section{Fourier Series and Spectral method}\label{fourierd}
In this section, we discuss the Fourier decomposition of the system \eqref{M1.N}. This analysis is essential to write the system
\eqref{M1.N} as an infinite system of ordinary differential equations. Then, we can consider Galerkin approximations extracted from that infinite system of ordinary differential equations.
Let us consider the three dimensional Euclidean space $\mathbb{R}^3$ with the standard basis $e_1,e_2,e_3$ which are given by by 
$$e_1=\begin{pmatrix}
      1\\0\\0
     \end{pmatrix},\quad e_2=\begin{pmatrix}
      0\\1\\0
     \end{pmatrix},\quad e_3=\begin{pmatrix}
      0\\0\\1
     \end{pmatrix}.
$$
We assume that the vector cross product $\times$ is defined in space $\mathbb{R}^3$. Then, by the properties of cross product in $\mathbb{R}^3$, we have that
\begin{align*}
e_1\times e_1&=e_2\times e_2=e_3\times e_3=0\in\mathbb{R}^3;\\
e_1\times e_2&=e_3=-e_2\times e_1\,,\,\, e_2\times e_3=e_1=-e_3\times e_2\,,\,\,e_3\times e_1=e_2=-e_1\times e_3.
\end{align*}
We know that $\left\{\lambda_n:=-n^2, \phi_n:=\cos(n\cdot)\right\}_{n\in \mathbb{N}_0}$ are the eigen pairs of the Neumann Laplace operator $-\mathcal{A}:=\Delta$ ($\mathcal{A}$ is defined above in \ref{op.n}) in the Hilbert space $\mathbb{L}^2(0,2\pi).$ Therefore,
Fourier expansion of solution $\M(t,x)$ can be written as
\begin{align}\label{f.expan.n}
\M(t,x)=\sum_{j=1}^3\sum_{n\in \mathbb{N}_0} m_n^j(t)\cos(nx)e_j.
\end{align}
Then,
\begin{align*}
\M_{xx}(t,x)=\sum_{j=1}^3\sum_{n\in \mathbb{N}_0}\lambda_n
 m_n^j(t)\cos(nx)e_j.
\end{align*}
We compute
\begin{align*}
&\M(t,x)\times \M_{xx}(t,x) \\
& =\left(\sum_{j=1}^3\sum_{n\in \mathbb{N}_0} m_n^j(t)\cos(nx)e_j\right)\times \left(\sum_{i=1}^3\sum_{k\in \mathbb{N}_0} m_k^j(t)\lambda_k \cos(kx)e_i\right),\\
& = \sum_{n\in \mathbb{N}_0} \sum_{k\in \mathbb{N}_0}\lambda_k \cos(nx)\cos(kx)
\left[\left(m_n^1(t)e_1+m_n^2(t)e_2+m_n^3(t)e_3\right)\times \left(m_k^1(t)e_1+m_k^2(t)e_2+m_k^3(t)e_3\right)\right],\\
& = \sum_{n\in \mathbb{N}_0} \sum_{k\in \mathbb{N}_0}\lambda_k \cos(nx)\cos(kx) \\
& \qquad
\left[\left(m_n^2(t)m_k^3(t)-m_n^3(t)m_k^2(t)\right)e_1
+\left(m_n^3(t)m_k^1(t)-m_n^1(t)m_k^3(t)\right)e_2
+ \left(m_n^1(t)m_k^2(t)-m_n^2(t)m_k^1(t)\right)e_3 \right].\\
\end{align*}
Similarly, we note that
\begin{align*}
&\M(t,x)\times(\M(t,x)\times \M_{xx}(t,x)) \\
& = \sum_{l\in \mathbb{N}_0}\sum_{n\in \mathbb{N}_0} \sum_{k\in \mathbb{N}_0}\lambda_k \cos(lx) \cos(nx)\cos(kx)  \\
& \qquad\qquad
\Big[\left\{m_l^2(t)\left(m_n^1(t)m_k^2(t)-m_n^2(t)m_k^1(t)\right)-m_l^3(t)\left(m_n^3(t)m_k^1(t)-m_n^1(t)m_k^3(t)\right)\right\}e_1 \\
& \qquad\qquad + \left\{m_l^3(t)\left(m_n^2(t)m_k^3(t)-m_n^3(t)m_k^2(t)\right)-m_l^1(t) \left(m_n^1(t)m_k^2(t)-m_n^2(t)m_k^1(t)\right)\right\}e_2 \\
& \qquad\qquad + \left\{m_l^1(t)\left(m_n^3(t)m_k^1(t)-m_n^1(t)m_k^3(t)\right) - m_l^2(t) \left(m_n^2(t)m_k^3(t)-m_n^3(t)m_k^2(t)\right)\right\}e_3 \Big]
\end{align*}
and
\begin{align*}
\M(t,x) \times \v(t,x)=\sum_{n\in\mathbb{N}_0}\sum_{j=1}^3\sum_{(k,l)\in \mathcal{K}^1}v_k^l(t)m_n^j(t)\varphi_i(x)\varphi_k(x)( e_j\times e_l).
\end{align*}
We use the following trigonometric identities
\begin{align*}
\int_0^{2\pi}\cos(nx)\cos(rx) dx & =\begin{cases}
                             2\pi,\,\, n=0=r, \\
                             \pi,\,\, n=r \neq 0,\\
                             0 ,\,\,\mbox{elsewhere},
                             \end{cases} \\
\int_0^{2\pi}\cos(nx)\cos(kx)\cos(rx)dx & \begin{cases}
                             2\pi, \,\, n+k=r=0, \\ 
                             \dfrac{\pi}{2},\,\, n+k=r \neq 0,\mbox{ or }  |n-k|=r\\
                             0 ,\,\,\mbox{elsewhere},
                             \end{cases} \\
\int_0^{2\pi}\cos(lx)\cos(nx)\cos(kx)\cos(rx)dx & =\begin{cases}
                             2\pi, \,\, n+k+l=r=0, \\ 
                             \dfrac{\pi}{4},\,\, n+k+l=r \neq 0, \mbox{ or } 
                             |n+k-l|=r,\mbox{ or } 
                             |n-k+l|=r,\\
                             \quad\mbox{ or } |k+l-n|=r, \\
                             0 ,\,\,\mbox{elsewhere}.
                             \end{cases} \\     
\end{align*}
Using the trigonometric identities above, we can write the control system \eqref{M1.N} as an infinite dimensional system of ordinary differential equations given below.
\begin{align}\label{inf.sys.o1}
\frac{d}{dt}m_i^1(t) & =\sum_{\substack{n+k=i,\\ |n-k|=i}}\lambda_k \left(m_n^2(t)m_k^3(t)-m_n^3(t)m_k^2(t)\right)\notag\\
& \quad - \sum_{\substack{n+k+l=i,\\ |n-k+l|=i,\\ |n+k-l|=i,\\|k+l-n|=i}}
\lambda_k \left\{m_l^2(t)\left(m_n^1(t)m_k^2(t)-m_n^2(t)m_k^1(t)\right)-m_l^3(t)\left(m_n^3(t)m_k^1(t)-m_n^1(t)m_k^3(t)\right)\right\}\notag\\
& \quad + \tilde{\mu}_1 \sum_{j=1}^3\sum_{\substack{n+k=i,\\|n-k|=i\\(k,l)\in \mathcal{K}^1}}v_k^l(t)m_n^j(t)\langle e_j\times e_l,e_1\rangle,
\end{align}
\begin{align}\label{inf.sys.o2}
\frac{d}{dt}m_i^2(t) & = \sum_{\substack{n+k=i,\\|n-k|=i}}\lambda_k \left(m_n^3(t)m_k^1(t)-m_n^1(t)m_k^3(t)\right)
\notag\\
& \quad - \sum_{\substack{n+k+l=i,\\|n-k+l|=i,\\ |n+k-l|=i,\\|k+l-n|=i}}
\lambda_k \left\{m_l^3(t)\left(m_n^2(t)m_k^3(t)-m_n^3(t)m_k^2(t)\right)-m_l^1(t) \left(m_n^1(t)m_k^2(t)-m_n^2(t)m_k^1(t)\right)\right\}\notag \\
& \quad + \tilde{\mu}_1 \sum_{j=1}^3\sum_{\substack{n+k=i,\\|n-k|=i\\(k,l)\in \mathcal{K}^1}}v_k^l(t)m_n^j(t)\langle e_j\times e_l,e_2\rangle,
\end{align}
and
\begin{align}\label{inf.sys.o3}
\frac{d}{dt}m_i^3(t) & = \sum_{\substack{n+k=i,\\|n-k|=i}}\lambda_k \left(m_n^1(t)m_k^2(t)-m_n^2(t)m_k^1(t)\right)
\notag\\
& \quad - \sum_{\substack{n+k+l=i,\\|n-k+l|=i,\\ |n+k-l|=i,\\|k+l-n|=i}}
\lambda_k  \left\{m_l^1(t)\left(m_n^3(t)m_k^1(t)-m_n^1(t)m_k^3(t)\right) - m_l^2(t) \left(m_n^2(t)m_k^3(t)-m_n^3(t)m_k^2(t)\right)\right\}\notag \\
& \quad + \tilde{\mu}_1\sum_{j=1}^3\sum_{\substack{n+k=i,\\|n-k|=i\\(k,l)\in \mathcal{K}^1}}v_k^l(t)m_n^j(t)\langle e_j\times e_l,e_3\rangle,
\end{align}
for $i,k,l,n\in\mathbb{N}_0.$


\section{Some Algebraic Identities}
 \label{app-A}
Here we list some algebraic identities regarding the scalar and vector products in $\mathbb{R}^3$ used in this paper. Assume that $a,b,c,d\in\mathbb{R}^3$. Then, the following identities hold
\begin{eqnarray} \label{eqn_A-1}
 a\times b&=&-b \times a,
\\
\label{eqn_A-2}
\lb a\times (b\times c),d\rb&=&\lb c,(d\times a)\times b\rb,
\\
\label{eqn_A-3}
\lb a\times b,c\rb&=&\lb b,c\times a\rb,
\\
\label{eqn_A-4}
\lb a\times b,b\rb&=&0, \\
\label{eqn_A-5}
 -\lb a\times b,c\rb&=&\lb b,a\times c\rb,\\
\label{eqn_A-6}
a\times(b\times c)&=&\lb a,c\rb b- \lb a,b\rb c,
\\
\label{eqn_A-7}
\vert a\times b\vert &\leq & \vert a\vert \vert b\vert.
\end{eqnarray}
In particular, if $\lb a,b\rb=0$, then $(a\times b)\times b= b \times(b\times a)=\lb b,a\rb b-\lb b,b\rb a=-\vert b\vert^2 a$ and
$a\times (a\times b)=\lb a,b\rb a-\lb a,a\rb b=-\vert a\vert^2 b$, i.e.
\begin{eqnarray}
\label{eqn_A-9}
(a\times b) \times b&=&-\vert b\vert^2 a, \, \text{if}\, \lb a,b\rb=0,
\\
\label{eqn_A-9b}
a\times (a\times b)&=& -\vert a\vert^2 b, \, \text{if}\, \lb a,b\rb=0.
\end{eqnarray}
Applying \eqref{eqn_A-3} and then \eqref{eqn_A-1} one can prove the following identity
\begin{eqnarray}
\lb a \times (a\times b),b\rb&=&-\vert a\times b\vert^2.
\end{eqnarray}

\end{appendix}


\end{document}